\pdfoutput=1
\RequirePackage{ifpdf}
\ifpdf 
\documentclass[pdftex]{sigma}
\else
\documentclass{sigma}
\fi

\usepackage{pst-plot}
\usepackage{mathdots}
\usepackage{pstricks}

\usepackage{multirow, blkarray}
\usepackage[all]{xy}

\numberwithin{equation}{section}

\newtheorem{thm}{Theorem}[section]
\newtheorem{conj}[thm]{Conjecture}
\newtheorem{open}[thm]{Open problem}
\newtheorem{lem}[thm]{Lemma}
\newtheorem{cor}[thm]{Corollary}
\newtheorem{prop}[thm]{Proposition}

\theoremstyle{definition}
\newtheorem{defn}[thm]{Definition}

\newtheorem{rem}[thm]{Remark}
\newtheorem{ex}[thm]{Example}

\newcommand{\Z}{\mathbb{Z}}
\newcommand{\C}{\mathbb{C}}

\newcommand{\A}{\mathcal{A}}

\newcommand{\Id}{\mathrm{Id}}
\newcommand{\GL}{\mathrm{GL}}
\newcommand{\SL}{\mathrm{SL}}
\newcommand{\Sp}{\mathrm{Sp}}

\newcommand{\Qc}{\mathcal{Q}}

\newcommand{\CP}{{\mathbb{CP}}}

\newcommand{\half}{\frac{1}{2}}
\newcommand{\thalf}{\frac{3}{2}}

\newcommand{\mA}{\mathrm{A}}

\newcommand{\mF}{\mathrm{F}}
\newcommand{\mC}{\mathrm{C}}
\newcommand{\mG}{\mathrm{G}}

\def\om{\omega}

\begin{document}
\allowdisplaybreaks

\newcommand{\arXivNumber}{1803.06001}

\renewcommand{\thefootnote}{}

\renewcommand{\PaperNumber}{089}

\FirstPageHeading

\ShortArticleName{Symplectic Frieze Patterns}

\ArticleName{Symplectic Frieze Patterns\footnote{This paper is a~contribution to the Special Issue on Algebra, Topology, and Dynamics in Interaction in honor of Dmitry Fuchs. The full collection is available at \href{https://www.emis.de/journals/SIGMA/Fuchs.html}{https://www.emis.de/journals/SIGMA/Fuchs.html}}}

\Author{Sophie MORIER-GENOUD}

\AuthorNameForHeading{S.~Morier-Genoud}

\Address{Sorbonne Universit\'e, Universit\'e Paris Diderot, CNRS, Institut de Math\'e\-ma\-tiques\\
de Jussieu-Paris Rive Gauche, IMJ-PRG, F-75005, Paris, France}
\Email{\href{mailto:sophie.morier-genoud@imj-prg.fr}{sophie.morier-genoud@imj-prg.fr}}
\URLaddress{\url{https://webusers.imj-prg.fr/~sophie.morier-genoud/}}

\ArticleDates{Received June 18, 2019, in final form November 07, 2019; Published online November 14, 2019}

\Abstract{We introduce a new class of friezes which is related to symplectic geometry. On the algebraic and combinatrics sides, this variant of friezes is related to the cluster algebras involving the Dynkin diagrams of type $\mC_{2}$ and $\mA_{m}$. On the geometric side, they are related to the moduli space of Lagrangian configurations of points in the 4-dimensional symplectic space introduced in~[Conley C.H., Ovsienko V., \textit{Math. Ann.} \textbf{375} (2019), 1105--1145]. Symplectic friezes share similar combinatorial properties to those of Coxeter friezes and $\SL$-friezes.}


\Classification{13F60; 05E10; 14N20; 53D30}

\rightline{\it To Dmitry Borisovich Fuchs on his 80th birthday}

\renewcommand{\thefootnote}{\arabic{footnote}}
\setcounter{footnote}{0}

\section{Introduction}

The notion of friezes goes back to Coxeter in the early 70's~\cite{Cox}. Friezes are arrays of numbers where neighboring values are related by a local arithmetic rule
modeled on the group $\SL_{2}$. Coxeter's friezes are surprisingly connected to many classical areas of mathematics such as projective geometry, number theory, enumerative combinatorics~\cite{CoCo, Cox}. Many variants of friezes have been recently studied in connection with more fields: cluster algebras, quiver representations, moduli spaces, integrable systems, algebraic combinatorics, see, e.g., \cite{ADSS,ARS,BaMa,BHJadv,MGOST,MGOTaif}.

In \cite{MGOST}, we establish a property called the ``triality'' that identifies three different objects: $\SL_{k}$-friezes, configurations of points in the projective spaces and periodic difference equations. A main application of the triality is that the combinatorics related to friezes allows to describe nice coordinate systems on the configurations spaces and in particular to exhibit a structure of cluster variety.

In \cite{CoOv}, Lagrangian configurations of lines in the symplectic spaces are identified with symmetric periodic linear difference equations. Valentin Ovsienko suggested to
complete the triality in the case of Lagrangian configurations with a combinatorial notion of ``symplectic friezes''. This is what we do in the present paper in the case of 4-dimensional symplectic spaces by introducing a new family of friezes that we call ``symplectic $2$-friezes''.

Lagrangian configurations in dimension 4 are interpreted as a discrete version of Legendrian knots. Symplectic $2$-friezes are a combinatorial interpretation of the space of Lagrangian configurations modulo $\Sp_{4}$-transformations. In particular they give special coordinate systems that provide a cluster variety structure on the moduli space of Lagrangian configurations.

The symplectic Lie group $\Sp_{4}$ appears in two ways. From an algebraic point of view, the local rule for the symplectic $2$-friezes is interpreted as cluster mutations involving the Dynkin diagram of type $\mC_{2}$ associated to the Lie group $\Sp_{4}$. From a geometric point of view, the symplectic $2$-friezes parametrize particular configurations of lines in the 4-dimensional symplectic space modulo the action of the Lie group $\Sp_{4}$.

Another particularity of the 4-dimensional case is that the ``local frieze rule'' used to define the symplectic 2-friezes involves $(2\times2)$-determinants instead of $4\times4$ expected ones. This rule mixes the rule of the 2-friezes of \cite{MGOTaif} and the rule of wall numbers \cite{CoGu}.

Symplectic 2-friezes enjoy nice properties which are very similar to those of the classical Coxeter friezes. In particular, they are invariant under the {glide symmetry} and therefore periodic. The periodicity of the friezes is a case of Zamolodchikov periodicity, see, e.g.,~\cite{Kel2}.

In Section \ref{mainth} we expose the main results of the paper: the combinatorial properties of the symplectic 2-friezes and the links with difference equations, symplectic geometry and cluster algebras. Section \ref{secex} contains examples illustrating the main results.

In Section \ref{secfrieq} we study the combinatorial structure of the symplectic 2-friezes and characterize them in terms of $\SL$-friezes. We explain the correspondence between symplectic $2$-friezes and symmetric difference equations. We also give explicit determinantal formulas for the entries in the friezes.

In Section \ref{sympgeo} we explain the link between the $2$-friezes and symplectic geometry. The results in this section are mainly based on \cite{CoOv}. We identify the symplectic $2$-friezes with Legendrian configurations in $\CP^{3}$ and discuss the relationship with symplectic forms.

Section \ref{cluster} makes the link between symplectic $2$-friezes and cluster algebras. The corresponding cluster algebras are generated from a product of Dynkin diagrams of type $\mC_{2}$ and $\mA_{w}$.

In Section \ref{integ}, we discuss open questions about the combinatorics of symplectic $2$-friezes with positive integer entries and about variants and more general symplectic friezes.

Appendixes~\ref{DJid} and~\ref{apSL} recollect important formulas and results on classical $\SL$-friezes from~\cite{BeRe} and~\cite{MGOST} which are extensively used throughout the paper.

\section{Main results}\label{mainth}

\subsection[Combinatorial description and characterizations of the symplectic 2-friezes]{Combinatorial description and characterizations\\ of the symplectic 2-friezes}\label{parsymp}

Roughly speaking, friezes may be defined as arrangements of numbers in a planar strip such that neighboring entries forming a given pattern always satisfy the same arithmetic relationship. For instance, the original Coxeter's friezes satisfy that every four adjacent entries $a$, $b$, $c$, $d$ forming a square are related by $ad-bc=1$, see \cite{Cox}.
Changing the arithmetic relation leads to different variants of friezes, see~\cite{MGblms} for a survey on the subject.

We introduce a new family of friezes. We call \textit{symplectic $2$-frieze} an array of complex numbers (or polynomials, rational functions, etc.) in the plane satisfying the following conditions:
\begin{itemize}\itemsep=0pt
\item the array has finitely many infinite rows, bounded top and bottom by a row of 1's;
\item the entries are alternatingly colored in black and white;
\item the colored entries are subject to the local rules
\begin{itemize}\itemsep=0pt
\item every white entry is equal to the $(2\times 2)$-determinant of the matrix formed by the four adjacent black entries;
 \item every square of a black entry is equal to the $(2\times 2)$-determinant of the matrix formed by the four adjacent white entries.
\end{itemize}
\end{itemize}
Symplectic 2-friezes are represented as follows
\begin{gather}\label{symp2fri}
\begin{array}{@{}ccccccccccccc@{}}
\cdots& \mathbf{1} & 1 & \mathbf{1} & 1 & \mathbf{1} & 1 & \mathbf{1} & 1 & \mathbf{1} & 1 & \cdots& \\
&\vdots & & \vdots & & \vdots & & \vdots & & \vdots & & \\
\cdots&\bullet & \circ & \bullet & \circ & \mathbf{B}&f&\bullet & \circ & \bullet & \circ & \cdots&\\
& \circ & \bullet & \circ & \mathbf{A}&e&\mathbf{D}&h&\bullet & \circ & \bullet \\
\cdots&\bullet & \circ & \bullet & \circ &\mathbf{C}&g&\bullet & \circ & \bullet & \circ & \cdots&\\
&\vdots & & \vdots & & \vdots & & \vdots & & \vdots & & \\
\cdots& \mathbf{1} & 1 & \mathbf{1} & 1 & \mathbf{1} & 1 & \mathbf{1} & 1 & \mathbf{1} & 1 & \cdots& \\
\end{array}
\end{gather}
in which the local rules read
\begin{gather*}
AD-BC=e, \qquad eh-fg=D^{2}, \qquad \ldots.
\end{gather*}

The \textit{width} of the frieze is the number of rows strictly between the top and bottom rows of~1's.

It is sometimes convenient to extend the array with additional rows of 0's above and below the bounding rows of 1's.

We will consider ``tame'' friezes which are friezes satisfying an extra condition of genericity (see Definition \ref{tame} for the details). For instance, friezes with no zero entries are all tame.

Symplectic $2$-friezes can be compared to the $2$-friezes of \cite{MGOTaif}. Recall that the latter are arrays as \eqref{symp2fri} with no color and in which the local rule reads with no squared values, i.e., for an ordinary $2$-frieze the rule would be the same everywhere $AD-BC=e$, $eh-fg=D$, $\dots$. Symplectic $2$-friezes can also be compared to the so called ``number walls'' for which the local rule in the array~\eqref{symp2fri} with no color would be $AD+BC=e^{2}$, $eh+fg=D^{2}$, see~\cite{CoGu} and J.~Propp's webpage for discussions on the subject.

The tame symplectic $2$-friezes share the same properties of symmetry as Coxeter friezes \cite{Cox} and as the $2$-friezes of \cite{MGOTaif}.

\begin{thm}\label{peri}All tame symplectic $2$-friezes of width $w$ are $2(w+5)$-periodic. Moreover, the arrays are all invariant under a glide reflection with respect to the median line.
\end{thm}

It turns out that the subarray consisting of the black entries of a tame symplectic $2$-frieze form a classical $\SL_{4}$-frieze. This allows the following characterization in terms of $\SL_{k}$-friezes.

\begin{thm}\label{SLcar}The set of tame symplectic $2$-friezes of width $w$ is in one-to-one correspondence with each of the following sets:
\begin{enumerate}\itemsep=0pt
\item[$1)$] the tame $\SL_{4}$-friezes of width $w$ in which every adjacent $(3\times 3)$-minors are equal to their central elements,
\item[$2)$] the tame $\SL_{4}$-friezes of width $w$ that are invariant under a glide reflection with respect to the median line,
\item[$3)$] the tame $\SL_{w+1}$-friezes of width $3$ that are symmetric with respect to the middle row.
\end{enumerate}

\end{thm}Theorems \ref{peri} and \ref{SLcar} are proved in Section~\ref{secproofth123}.

\subsection{Symplectic 2-friezes and difference equations}

Consider the system of linear difference equations of the form
\begin{gather}\label{recur}
V_{i}=a_i V_{i-1}-b_i V_{i-2} + a_{i-1}V_{i-3}-V_{i-4},
\end{gather}
for $i\in \Z$, where $a_i,b_i\in\C$ are given coefficients, and $(V_{i})_{i}$ is a sequence of indeterminates or a~``solution''.

Following \cite{Kri,MGOST}, we call such difference equation \textit{$n$-superperiodic} if
\begin{itemize}\itemsep=0pt
\item all the coefficients are $n$-periodic, i.e., $a_{i+n}=a_{i}$, $b_{i+n}=b_{i}$, for all $i\in \Z$,
\item all solutions $(V_{i})_{i\in \Z}$ of the system are $n$-antiperiodic, i.e., $V_{i+n}=-V_{i}$ for all $i\in \Z$.
\end{itemize}

\begin{thm}\label{recuriso}There is a one-to-one correspondence between the set of tame symplectic $2$-friezes of width $w=n-5$ and the set of $n$-superperiodic difference~equations of the form~\eqref{recur}.
\end{thm}
The correspondence is given explicitly: the entries in the first row of the tame symplectic $2$-frieze give the coefficients of the superperiodic equation and \textit{vice versa}, see Proposition~\ref{propfrieq} from which Theorem~\ref{recuriso} is a direct consequence.

\subsection{The algebraic variety of symplectic 2-friezes}

One can give explicit algebraic conditions on the coefficients $a_{i}$ and $b_{i}$ of the difference equation~\eqref{recur} in order to have the superperiodicity property. This has been done in~\cite{MGOST} for general linear difference equations, and in~\cite{CoOv} for the particular case of equations with symmetric coefficients. The conditions comes from the equation $M=-\Id$ where $M$ is a monodromy matrix associated to the difference equations.

In our situation, $M$ is a matrix of $\Sp_{4}(\C)$ and its entries can be expressed with the help of the following determinants
\begin{gather*}
\Delta_{i,j}=
\left|
\begin{array}{@{}ccccccccccccc@{}}
a_{i}& b_{i+1} & a_{i+1} &1 \\
1& a_{i+1}& b_{i+2} & a_{i+2} &1\\
 & \ddots & \ddots &\ddots &\ddots &\ddots \\
&& 1& a_{j-2}& b_{j-1} & a_{j-1} \\
 &&&1& a_{j-1}& b_{j} \\
&&&&1&a_{j}
\end{array}
\right|.
\end{gather*}

This allows us to obtain a complete system of polynomial equations for the set of tame symplectic $2$-friezes of width $w$ over the complex numbers as an algebraic subvariety of $\C^{2n}$, where $n=w+5$.

\begin{prop}\label{eqvar} Let $(a_{i},b_{i})_{1\leq i \leq n}$ be coordinates on $\C^{2n}$.
The set of tame symplectic $2$-friezes of width $w=n-5$ is the algebraic subvariety of $\C^{2n}$ of dimension $2w$ given by the system of $10$ equations $($with the convention $a_{0}=a_{n})$
\begin{gather}\label{systvar}
\begin{cases}
\Delta_{3, n-3} = a_{n},& \\
\Delta_{k+2, n-3+k} = 1,& k=0, 1, \\
\Delta_{k+1, n-3+k} = 0,& k=0, 1, 2, \\
\Delta_{k, n-3+k} = 0,& k=0, 1, 2, 3.
\end{cases}
\end{gather}
\end{prop}
The proof is given in Section~\ref{secproofprop}.

\subsection{The cluster structure of the variety of tame symplectic 2-friezes}
Two consecutive columns in a 2-frieze give a system of $2w$ coordinates on the variety of symplectic 2-friezes. The transition maps between the different systems of coordinates are interpreted as sequences of cluster mutations inside a cluster algebra. More precisely one has the following result.

\begin{thm}\label{thmclust} The variety of tame symplectic $2$-friezes of width $w=n-5$ contains as an open dense subset the cluster variety associated to an orientation of the product of Dynkin diagrams $\mC_{2}\times \mA_{w}$.
\end{thm}

This result is reformulated more precisely as Propositions~\ref{thmAmas2f} and~\ref{clustparam} which are proved in Sections~\ref{clustvar1} and~\ref{clustvar2}.

\subsection{Link with symplectic geometry}\label{parngon}

In \cite{CoOv} the variety of $n$-superperiodic linear difference equations of order $2k$, with symmetric coefficients, has been identified with the moduli space
$\mathcal{L}_{2k,n}$ of Lagrangian configurations of $n$ lines in $\C^{2k}$. Applying the results and ideas of \cite{CoOv} in the case $k=2$ gives a nice geometric interpretation of the symplectic $2$-friezes in a symplectic space.

We present this geometric interpretation in more naive terms where we prefer the projective version of the Lagrangian configurations.

\begin{defn}[\cite{CoOv}, \textit{Lagrangian configurations of lines}]\label{defngon} Consider the projective space $\CP^{3}$ equipped with a contact structure given by the data of hyperplanes $(H_{v})_{v\in \CP^{3}}$ attached to each point of the space. We call a \textit{Legendrian configuration} or \textit{Legendrian $n$-gon}, every sequence $(v_{i})_{i\in \Z}$ of points in $\CP^{3}$ such that:
\begin{itemize}\itemsep=0pt
\item $v_{i+n}=v_{i}$, for all $i\in \Z$,
\item $v_{i-1}$ and $v_{i+1}$ belong to $H_{v_{i}}$ for all $i\in \Z$.
\end{itemize}
\end{defn}

Legendrian $n$-gons are discrete analogues of Legendrian knots.

Generically, the vertex $v_{i+2}$ of a Legendrian $n$-gon does not belong to $H_{v_{i}}$. We will consider the moduli space of generic Legendrian configurations modulo $\mathrm{P}\Sp_{4}$-equivalence.

\begin{thm}\label{thmngonfri}For odd $n\geq 5$, the space of tame symplectic $2$-friezes of width $w=n-5$ over the complex numbers is isomorphic to the moduli space of generic Legendrian $n$-gons in $\CP^{3}$.
\end{thm}

The correspondence is given in details in Section~\ref{sympgeo}.

\section{Examples}\label{secex}

In this section we work out examples to illustrate the results presented in the previous section.

We start by giving examples of symplectic $2$-friezes as described in Section~\ref{parsymp}.

\begin{ex}(a) Tame symplectic $2$-friezes of width 1 with integer entries
\begin{gather}\label{fw10}
\begin{array}{@{}rrrrrrrrrrrrrr@{}}
\cdots&1&\mathbf{1}&1&\mathbf{1}&1&\mathbf{1}&1&\mathbf{1}&1&\mathbf{1}&1&\mathbf{1}&\cdots\\
\cdots&\mathbf{0}&-1&\mathbf{1}&-2&\mathbf{-1}&-1&\mathbf{0}&-1&\mathbf{1}&-2&\mathbf{-1}&-1 &\cdots\\
\cdots&1&\mathbf{1}&1&\mathbf{1}&1&\mathbf{1}&1&\mathbf{1}&1&\mathbf{1}&1&\mathbf{1}&\cdots \\
\end{array}\\
\label{friw1}
\begin{array}{@{}ccccccccccccccccccc@{}}
\cdots&1&\mathbf{1}&1&\mathbf{1}&1&\mathbf{1}&1&\mathbf{1}&1&\mathbf{1}&1&\mathbf{1}&\cdots\\
\cdots&\mathbf{1}&2&\mathbf{3}&5&\mathbf{2}&1&\mathbf{1}&2&\mathbf{3}&5&\mathbf{2}&1 &\cdots\\
\cdots&1&\mathbf{1}&1&\mathbf{1}&1&\mathbf{1}&1&\mathbf{1}&1&\mathbf{1}&1&\mathbf{1}&\cdots \end{array}\end{gather}

(b) Tame symplectic $2$-friezes of width 1 with positive real entries and with complex entries
\begin{gather*}
\begin{array}{@{}ccccccccccccccccccc@{}}
\cdots&1&\mathbf{1}&1&\mathbf{1}&1&\mathbf{1}&1&\mathbf{1}&1&\mathbf{1}&1&\mathbf{1}&\cdots\\
\cdots&\mathbf{2\sqrt2}&3&\mathbf{\sqrt2}&1&\mathbf{\sqrt2}&3&\mathbf{2\sqrt2}&3&\mathbf{\sqrt2}&1&\mathbf{\sqrt2}&3 &\cdots\\
\cdots&1&\mathbf{1}&1&\mathbf{1}&1&\mathbf{1}&1&\mathbf{1}&1&\mathbf{1}&1&\mathbf{1}&\cdots \end{array}\\
\begin{array}{@{}rrrrrrrrrrrrrrrr@{}}
\cdots&1&\mathbf{1}&1&\mathbf{1}&1&\mathbf{1}&1&\mathbf{1}&1&\mathbf{1}&1&\mathbf{1}&\cdots\\
\cdots&\mathbf{i}&1&\mathbf{-2i}&-3&\mathbf{-i}&0&\mathbf{i}&1&\mathbf{-2i}&-3&\mathbf{-i}&0&\cdots\\
\cdots&1&\mathbf{1}&1&\mathbf{1}&1&\mathbf{1}&1&\mathbf{1}&1&\mathbf{1}&1&\mathbf{1}&\cdots \end{array}\end{gather*}

(c) Tame symplectic $2$-friezes of width 2 and 3 with positive integer entries
\begin{gather}\label{friw2}
\begin{array}{@{}ccccccccccccccccccc@{}}
\cdots&1&\mathbf{1}&1&\mathbf{1}&1&\mathbf{1}&1&\mathbf{1}&1&\mathbf{1}&1&\mathbf{1}&1&\mathbf{1}& 1&\cdots\\
\cdots&\mathbf{6}&14&\mathbf{3}&1&\mathbf{1}&2&\mathbf{3}&6&\mathbf{4}&5&\mathbf{2}&1&\mathbf{1}&3&\mathbf{6 } &\cdots\\
\cdots&6&\mathbf{4}&5&\mathbf{2}&1&\mathbf{1}&3&\mathbf{6 } &14&\mathbf{3}&1&\mathbf{1}&2&\mathbf{3}&6&\cdots\\
\cdots&\mathbf{1}&1&\mathbf{1}&1&\mathbf{1}&1&\mathbf{1}&1&\mathbf{1}&1&\mathbf{1}&1&\mathbf{1}&1&\mathbf{1}&\cdots \end{array}\\
\label{friw3}
\begin{array}{@{}ccccccccccccccccccccccc@{}}
\cdots&\mathbf{1}&1&\mathbf{1}&1&\mathbf{1}&1&\mathbf{1}&1&\mathbf{1}&1&\mathbf{1}&1&\mathbf{1}& 1&\mathbf{1}&1&\mathbf{1}&1&\cdots\\
\cdots&1&\mathbf{2}&5&\mathbf{4}&6&\mathbf{4}&6&\mathbf{3}&2&\mathbf{1}&1&\mathbf{4}&30&\mathbf{10 } & 4&\mathbf{1}&1&\mathbf{2}&\cdots\\
\cdots&\mathbf{1}&1&\mathbf{3}&14&\mathbf{10}&20&\mathbf{6 } &3&\mathbf{1}&1&\mathbf{3}&14&\mathbf{10}&20&\mathbf{6 } &3&\mathbf{1}&1&\cdots\\
\cdots&2&\mathbf{1}&1&\mathbf{4}&30&\mathbf{10 } & 4&\mathbf{1}&1&\mathbf{2}&5&\mathbf{4}&6&\mathbf{4}&6&\mathbf{3}&2&\mathbf{1}&\cdots\\
\cdots&\mathbf{1}&1&\mathbf{1}&1&\mathbf{1}&1&\mathbf{1}&1&\mathbf{1}&1&\mathbf{1}&1&\mathbf{1}&1&\mathbf{1}&1&\mathbf{1}&1&\cdots \end{array}\end{gather}
\end{ex}

\begin{ex}[glide reflection]
Recall that a glide reflection is the composition of a reflection about a line and a translation along that line. The invariance mentioned in Theorem \ref{peri} can be easily observed in the above examples. For instance, we can display a fundamental domain in the array~\eqref{friw3} that repeats under a glide reflection. Note that the glide symmetry implies the periodicity of the arrays.
\begin{figure}[hbtp]\centering

\setlength{\unitlength}{1444000sp}%
\begin{picture}(-18,4.015)(16.810606,-2.015)
\put(0.41060662,1.6){$1$}
\put(1.2106066,1.6){$1$}
\put(2.0106065,1.6){$1$}
\put(2.8106067,1.6){$1$}
\put(3.6106067,1.6){$1$}
\put(4.4106064,1.6){$1$}
\put(5.2106066,1.6){$1$}
\put(6.010607,1.6){$1$}
\put(6.8106065,1.6){$1$}
\put(7.6106067,1.6){$1$}
\put(8.410606,1.6){$1$}
\put(9.210607,1.6){$1$}
\put(10.010607,1.6){$1$}
\put(10.810607,1.6){$1$}
\put(11.610606,1.6){$1$}
\put(12.410606,1.6){$1$}
\put(13.210607,1.6){$1$}
\put(14.010607,1.6){$1$}
\put(0.41060662,0.8){$1$}
\put(1.2106066,0.8){$2$}
\put(2.0106065,0.8){$5$}
\put(2.8106067,0.8){$4$}
\put(3.6106067,0.8){$6$}
\put(4.4106064,0.8){$4$}
\put(5.2106066,0.8){$6$}
\put(6.010607,0.8){$3$}
\put(6.8106065,0.8){$2$}
\put(7.6106067,0.8){$1$}
\put(8.410606,0.8){$1$}
\put(9.210607,0.8){$4$}
\put(10.010607,0.8){$30$}
\put(10.810607,0.8){$10$}
\put(11.610606,0.8){$4$}
\put(12.410606,0.8){$1$}
\put(13.210607,0.8){$1$}
\put(14.010607,0.8){$2$}
\put(0.41060662,0.0){$1$}
\put(1.2106066,0.0){$1$}
\put(2.0106065,0.0){$3$}
\put(2.8106067,0.0){$14$}
\put(3.6106067,0.0){$10$}
\put(4.4106064,0.0){$20$}
\put(5.2106066,0.0){$6$}
\put(6.010607,0.0){$3$}
\put(6.8106065,0.0){$1$}
\put(7.6106067,0.0){$1$}
\put(8.410606,0.0){$3$}
\put(9.210607,0.0){$14$}
\put(10.010607,0.0){$10$}
\put(10.810607,0.0){$20$}
\put(11.610606,0.0){$6$}
\put(12.410606,0.0){$3$}
\put(13.210607,0.0){$1$}
\put(14.010607,0.0){$1$}
\put(0.41060662,-0.79999995){$2$}
\put(1.2106066,-0.79999995){$1$}
\put(2.0106065,-0.79999995){$1$}
\put(2.8106067,-0.79999995){$4$}
\put(3.6106067,-0.79999995){$30$}
\put(4.4106064,-0.79999995){$10$}
\put(5.2106066,-0.79999995){$4$}
\put(6.010607,-0.79999995){$1$}
\put(6.8106065,-0.79999995){$1$}
\put(7.6106067,-0.79999995){$2$}
\put(8.410606,-0.79999995){$5$}
\put(9.210607,-0.79999995){$4$}
\put(10.010607,-0.79999995){$6$}
\put(10.810607,-0.79999995){$4$}
\put(11.610606,-0.79999995){$6$}
\put(12.410606,-0.79999995){$3$}
\put(13.210607,-0.79999995){$2$}
\put(14.010607,-0.79999995){$1$}
\put(0.41060662,-1.5999999){$1$}
\put(1.2106066,-1.5999999){$1$}
\put(2.0106065,-1.5999999){$1$}
\put(2.8106067,-1.5999999){$1$}
\put(3.6106067,-1.5999999){$1$}
\put(4.4106064,-1.5999999){$1$}
\put(5.2106066,-1.5999999){$1$}
\put(6.010607,-1.5999999){$1$}
\put(6.8106065,-1.5999999){$1$}
\put(7.6106067,-1.5999999){$1$}
\put(8.410606,-1.5999999){$1$}
\put(9.210607,-1.5999999){$1$}
\put(10.010607,-1.5999999){$1$}
\put(10.810607,-1.5999999){$1$}
\put(11.610606,-1.5999999){$1$}
\put(12.410606,-1.5999999){$1$}
\put(13.210607,-1.5999999){$1$}
\put(14.010607,-1.5999999){$1$}
\thicklines
\put(0.110606613,1.6){\line(1,-1){3.6}}
\put(3.7106067,-2.0){\line(1,0){2.35}}
\put(6.010607,-2.0){\line(1,1){4}}
\put(10.010607,2.0){\line(1,0){2.5}}
\put(12.510606,2.0){\line(1,-1){4}}
\put(14.810607,1.6){$1$}
\put(15.610606,1.6){$1$}
\put(16.410606,1.6){$1$}
\put(14.810607,-1.5999999){$1$}
\put(15.610606,-1.5999999){$1$}
\put(16.410606,-1.5999999){$1$}
\put(14.810607,0.8){$5$}
\put(15.610606,0.8){$4$}
\put(16.410606,0.8){$6$}
\put(14.810607,0.0){$3$}
\put(15.610606,0.0){$14$}
\put(16.410606,0.0){$10$}
\put(14.810607,-0.79999995){$1$}
\put(15.610606,-0.79999995){$4$}
\put(16.410606,-0.79999995){$30$}
\end{picture}
\caption{Glide symmetry in the frieze \eqref{friw3}.}\label{glide}
\end{figure}
\end{ex}

\begin{ex}[corresponding $\SL$-friezes] We illustrate Theorem~\ref{SLcar}.
The $\SL_{4}$-frieze corresponding to a symplectic 2-frieze is simply given by the subarray of black entries. Applying the combinatorial Gale duality (see Appendix \ref{friezeG}) one obtains the corresponding $\SL_{w+1}$-frieze of width~3.

\begin{figure}[hbtp]\centering
\input{sympFriezeGaled}
\caption{The $\SL_3$-frieze of width 3 (left) and the $\SL_4$-frieze of width 2 (right) corresponding to the symplectic 2-frieze~\eqref{friw2}.}\label{DualFriezes}
\end{figure}

The $\SL$-friezes corresponding to the symplectic 2-frieze \eqref{friw2} are given in Fig.~\ref{DualFriezes}. The friezes are related by Gale duality. One can check that
$WD\big({}^tV\big)=0$ where $W$ and $V$ are respectively the $3\times 7$ and $4\times 7$ matrices defined on Fig.~\ref{DualFriezes}, and $D$ the $7\times 7$ diagonal matrix with diagonal coefficients $(1,-1,1,-1,1,-1,1)$.
\end{ex}

\begin{ex}[corresponding difference equation] We illustrate Theorem~\ref{recuriso}.
The array \eqref{friw2} provides the following 7-periodic sequences of coefficients:
\begin{gather*}
(\ldots, a_{0}, a_{1}, \ldots, a_{6}, \ldots)=(\ldots, 6,3,1,3,4,2,1, \ldots),\\
(\ldots, b_{0}, b_{1}, \ldots ,b_{6}, \ldots)=(\ldots, 3,14,1,2,6,5,1, \ldots).
\end{gather*}
Let us check that the associated recurrence \eqref{recur} is indeed $7$-antiperiodic.
Choose the initial values $(V_{-3}, V_{-2}, V_{-1}, V_{0})=(x,y,z,t)$ and compute
\begin{gather*}
V_{1} = 3V_{0}-14V_{-1}+6V_{-2}-V_{-3} = 3t-14z+6y-x,\\
V_{2} = V_{1}-V_{0}+3V_{-1}-V_{-2} = 2t-11z+5y-x,\\
V_{3} = 3V_{2}-2V_{1}+V_{0}-V_{-1} = t-6z+3y-x,\\
V_{4} = 4V_{3}-6V_{2}+3V_{1}-V_{0} = -x = -V_{-3},\\
V_{5} = 2V_{4}-5V_{3}+4V_{2}-V_{1} = -y = -V_{-2},\\
V_{6} = V_{5}-V_{4}+2V_{3}-V_{2} = -z = -V_{-1},\\
V_{7} = 6V_{6}-3V_{5}+V_{4}-V_{3} = -t = -V_{0}.
\end{gather*}
\end{ex}

\begin{ex}[equations of the variety]\label{exeqvar}
For $w=1$ the system of equations \eqref{systvar} is
\begin{alignat*}{5}
& a_{3}=a_{6}, \ \ \ \ \ \ \ \ \; (1) \qquad && a_{1}-a_{3}b_{2}+a_{2} = 0, \ (4) \quad && b_{1}-a_{1}a_{3}+1= 0, \ (7) & \\
& a_{3}a_{2}-b_{3} = 1, \ (2) \qquad && a_{2}-a_{4}b_{3}+a_{3} = 0 ,\ (5) \qquad && b_{2}-a_{2}a_{4}+1= 0, \ (8) &\\
& a_{4}a_{3}-b_{4} = 1, \ (3) \qquad && a_{5}-a_{3}b_{5}+a_{4} = 0, \ (6) \qquad && b_{5}-a_{4}a_{2}+1= 0, \ (9) & \\
 & && && b_{6}-a_{5}a_{3}+1 = 0. \ (10) &
\end{alignat*}
Note that using the first row expansion or last column expansion in $\Delta_{i,j}$ the subsets of equations in the rows of the system~\eqref{systvar} can be immediately simplified using the previous subsets, in order to decrease the degrees of the equations.

The dimension of the variety is 2. The above system can be solved generically using two parameters $(a,b)\not=(0,0)$. E.g., choosing $(a,b)=(a_{3},b_{3})$ we deduce $b_{4}$ from $b_{3}b_{4}$ using a~combination of equations number~3, 5 and 2:
\begin{gather*}
b_{3}b_{4}=b_{3}(a_{4}a_{3}-1)=b_{3}a_{4}a_{3}-b_{3}=(a_{2}+a_{3})a_{3}-b_{3}=a_{3}^{2}+(a_{2}a_{3}-b_{3})=a_{3}^{2}+1.
\end{gather*}
So that $b_{4}=\frac{1+a_{3}^{2}}{b_{3}}$. Then all the variables are deduced one after each other using only one equation of the system, in the following order
\begin{alignat*}{4}
& a_{2}= \frac{1+b_{3}}{a_{3}},\qquad && a_{4} = \frac{1+b_{4}}{a_{3}}, \qquad && & \\
& b_{2}= a_{2}a_{4}-1, \qquad && a_{1} = a_{3}b_{2}-a_{2},\qquad & & b_{1} =a_{1}a_{3}-{1},& \\
& b_{5}= a_{2}a_{4}-1,\qquad && a_{5} = a_{3}b_{5}-a_{4},\qquad && b_{6} = a_{5}a_{3}-1,&\\
& a_{6}= a_{3}.\qquad && && &
\end{alignat*}
Expressing everything in terms of $(a,b)=(a_{3},b_{3})$, one gets
\begin{alignat*}{3}
& b_{1}= b_{4} = \frac{1+a^{2}}{b}, \qquad & & a_{1} = a_{4} = \frac{1+b+a^{2}}{ab},& \\
& b_{2}= b_{5} = \frac{(1+b)^{2}+a^{2}}{a^{2}b}, \qquad & & a_{2} = a_{5} = \frac{1+b}{a}, & \\
& b_{3}= b_{6} = b, \qquad & & a_{3} = a_{6} = a.
\end{alignat*}
\end{ex}

\begin{ex}[corresponding Legendrian $n$-gon] We illustrate Theorem~\ref{thmngonfri}. We can associate a Legendrian heptagon to the $2$-frieze of width 2 given in~\eqref{friw2}.
We choose a $4\times 7$-subarray of black entries in the frieze, e.g., the block displayed in Fig.~\ref{DualFriezes}, and denote by $V_{i}$ its columns, so that
\begin{gather*}
\begin{blockarray}{cccccccc}
V_{1} & V_{2} &V_{3} &V_{4} &V_{5} &V_{6} &V_{7} \\[1pt]
\begin{block}{(ccccccc)c}
1&4&3&1&0&0&0 \\[1pt]
 0&1&2&1&1&0&0\\[1pt]
 0 &0&1 &1&3&1&0\\[1pt]
0&0&0&1&6&4&1\\[1pt]
\end{block}
\end{blockarray}
\end{gather*}
We extend the sequence by periodicity $V_{i+7}=V_{i}$.

Consider the symplectic form on $\C^{4}$ given by
\begin{gather*}\om=
\left(
\begin{matrix}
 \hphantom{-}0&\hphantom{-}0 & \hphantom{-}1& \hphantom{-}0\\
\hphantom{-}0&\hphantom{-}0 &-4 &\hphantom{-}1\\
-1 &\hphantom{-}4 &\hphantom{-}0&\hphantom{-}0 \\
\hphantom{-}0&-1&\hphantom{-}0&\hphantom{-}0
\end{matrix}\right).
\end{gather*}
Attach to each $V_{i}$ the hyperplane $H_{i}:=\{V_{i}\}^{\perp_{\om}}$ formed by the orthogonal vectors. One can easily check that $V_{i-1}$, $V_{i+1}$ belong to $H_{i}$ for all~$i$. The 7-periodic sequence $v=(\C V_{i})_{i}$ forms a~Legendrian heptagon in~$\CP^{3}$.

From the Legendrian $n$-gon, one can compute the entries in the corresponding frieze using the values of
\begin{gather*}
\omega (V_{i}, V_{j}).
\end{gather*}
E.g., in this example the black entries of the first row (which is the same as the second row by glide symmetry) of~\eqref{friw2} is $\omega(V_{2}, V_{5}), \omega(V_{3}, V_{6}), \omega(V_{4},V_{7}), \ldots$.

If we move the initial $(4\times 7)$-subarray along the north-east diagonal, the sequence changes by a shift of indices. If we move the initial $(4\times 7)$-subarray along the north-west diagonal, the sequence changes under the action of $\Sp_{4}$. Therefore any choice of $(4\times 7)$-subarrays leads to the same Legendrian $n$-gon, modulo a shift of indices of the vertices and modulo $\Sp_{4}$-action.
\end{ex}

\begin{ex}[cluster variables in the frieze]Every symplectic $2$-frieze is generically determined by the entries in 2 consecutive columns. All the other entries can be expressed as rational fractions of the initial entries using the frieze rule. These expressions can be recognized as cluster variables of type $\mC_{2}\times A_{w}$, where $w$ is the width of the frieze. For instance, a frieze of width $w=1$ will produce the 6 cluster variables of type $\mC_{2}$:
\begin{gather}\label{frizC2}
\begin{array}{ccccccccccccccccccc}
\cdots&\mathbf{1}&1&\mathbf{1}&1&\mathbf{1}&1&\mathbf{1}&1&\cdots\\[2pt]
\cdots&x_{1}&\mathbf{x_{2}}&\frac{1+x_{2}^{2}}{x_{1}}&\mathbf{\frac{1+x_{1}+x_{2}^{2}}{x_{1}x_{2}}}&\frac{(1+x_{1})^{2}+x_{2}^{2}}{x_{1}x_{2}^{2}}&\mathbf{\frac{1+x_{1}}{x_{2}}}&x_{1}&\mathbf{x_{2}} &\cdots\\[2pt]
\cdots&\mathbf{1}&1&\mathbf{1}&1&\mathbf{1}&1&\mathbf{1}&1&\cdots
\end{array}\end{gather}
\end{ex}

\section{Symplectic 2-friezes, SL-friezes and difference equations}\label{secfrieq}

\subsection{Notation for the symplectic 2-friezes}\label{nota}

We complete the description of the 2-friezes given in Section~\ref{parsymp} by introducing a labelling of the entries and more convention.

\begin{figure}[t]\centering
 $
 \xymatrix
 @!0 @R=1.6cm @C=1.6cm
 {
 &
&\ar@{-}[rd]
& d_{i-\thalf,j+\thalf}\ar@{--}[ld]\ar@{--}[rd]
&\ar@<2pt>@{-}[ld]
&\\
&\ldots \ar@<2pt>@{-}[rd]
&d_{i-\thalf,j+\half}\ar@{--}[rd]\ar@{--}[ld]\ar@{--}[lu]
& d_{i-1,j+1}\ar@{-}[ld]\ar@{-}[rd]
&d_{i-\half,j+\thalf}\ar@{--}[rd]\ar@{--}[ld]\ar@{--}[ru]
&\ar@{-}[ld]\\
&d_{i-\thalf,j-\half} \ar@{--}[rd]
& d_{i-1,j}\ar@<2pt>@{-}[rd]\ar@<2pt>@{-}[ld]
&d_{i-\half,j+\half}\ar@{--}[ld]\ar@{--}[rd]
& d_{i,j+1}\ar@<2pt>@{-}[ld]\ar@{-}[rd]
&\ar@{--}[ld] \cdots\\
&d_{i-1,j-1} \ar@{-}[rd]
&d_{i-\half,j-\half} \ar@{--}[rd] \ar@{--}[ld]
&d_{i,j}\ar@{-}[ld]\ar@{-}[rd]
&d_{i+\half,j+\half}\ar@{--}[ld]\ar@{--}[rd]
&d_{i+1,j+1}\ar@{-}[ld]
 \\
& &d_{i,j-1}\ar@{-}[rd]\ar@{-}[ld]
& d_{i+\half,j-\half}\ar@{--}[ld]\ar@{--}[rd]
&d_{i+1,j} \ar@{-}[ld]\ar@{-}[rd]&\\
&&&d_{i+1,j-1}&&&&&
}
$
\caption{Coordinate labelling in a symplectic $2$-frieze. Black entries are connected with plain lines and the white entries with dashed lines.}\label{frizind}
\end{figure}
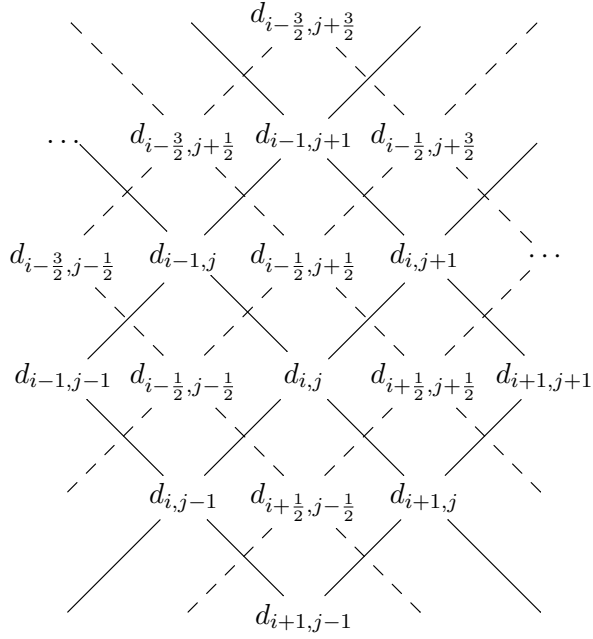

The entries in a symplectic $2$-frieze of width $w$ are denoted by $(d_{i,j})$, ${i,j\in\Z}$ for the ``black entries'' and $\big(d_{i+\half,j+\half}\big)$, $i,j\in\Z$ for the ``white entries''. The local rules around a white entry is
\begin{gather*}
d_{i+\half,j+\half}=d_{i,j}d_{i+1,j+1}-d_{i+1,j}d_{i,j+1},
\end{gather*}
and the local rule around a black entry is
\begin{gather*}
 d_{i,j}^{2}=d_{i-\half,j-\half}d_{i+\half,j+\half}-d_{i+\half,j-\half}d_{i-\half,j+\half}.
\end{gather*}
The
 boundary conditions are
\begin{itemize}\itemsep=0pt
\item $d_{i,i-1}=d_{i+\half,i-1+ \half}=1$,
\item $d_{i,i+w}=d_{i+\half,i+w+ \half}=1$.
\end{itemize}
Note that the index set for the entries is
\begin{gather*}
\{(i,j) \,|\, i\in \Z, \, i-1\leq j \leq i+w\}.
\end{gather*}
However, we will often consider the $2$-friezes as infinite arrays by adding the extra conventional conditions:
\begin{itemize}\itemsep=0pt
\item $d_{i,i-\ell}=d_{i+\half,i-\ell +\half}=0$, for $\ell=2,3,4$,
\item $d_{i,j+w+5}=-d_{i,j}$ for all $i,j$,
\end{itemize}
which means that we add three rows of zeros at the top and the bottom of the frieze and extend it to an antiperiodic infinite array.

Fig.~\ref{frizind} shows how the entries are organized in the plane. Note that a row in the frieze consists in the entries $d_{i,j}$ where $j-i$ is a fixed constant whereas $i$ fixed or $j$ fixed form the diagonals of the frieze.

\subsection{The tameness condition}
We define the notion of ``tame'' 2-friezes. This notion is an analog of the tameness condition for the classical $\SL$-friezes, see Appendix~\ref{tameSL}. It is understood as a genericity condition extending the properties satisfied by the friezes with no zero entries, see Proposition~\ref{3344} and Remark~\ref{remtame}.

\begin{prop}\label{3344}A symplectic $2$-frieze containing no zero entries satisfies the following pro\-per\-ties.
\begin{enumerate}\itemsep=0pt
\item[$(i)$] All $(3\times 3)$-minors of adjacent black entries are equal to their central elements, i.e., for all $ i,j\in \Z$
\begin{equation}\label{C1}
\left|
\begin{matrix}
 d_{i-1,j-1} & d_{i-1,j} &d_{i-1,j+1} \\
 d_{i,j-1}& d_{i,j} &d_{i,j+1} \\
d_{i+1,j-1} & d_{i+1,j} & d_{i+1,j+1}
\end{matrix}
\right|= d_{i,j}.
\end{equation}

\item[$(ii)$] All $(4\times 4)$-minors of adjacent black entries are equal to $1$, i.e., for all $ i,j\in \Z$
\begin{gather}\label{C2}
\left|
\begin{matrix}
 d_{i-1,j-1} & d_{i-1,j} &d_{i-1,j+1} &d_{i-1,j+2} \\
 d_{i,j-1}& d_{i,j} &d_{i,j+1} &d_{i,j+2} \\
d_{i+1,j-1} & d_{i+1,j} & d_{i+1,j+1}&d_{i+1,j+2}\\
d_{i+2,j-1} & d_{i+2,j} & d_{i+2,j+1}&d_{i+2,j+2}
\end{matrix}\right|=1.
\end{gather}
\item[$(iii)$] All $(5\times 5)$-minors of adjacent black entries are equal to $0$, i.e., for all $ i,j\in \Z$
\begin{gather}\label{C3}
\left|\begin{matrix}
 d_{i-1,j-1} & d_{i-1,j} &\ldots &d_{i-1,j+3} \\
 d_{i,j-1}& d_{i,j} &\ldots & \vdots \\
\vdots & \vdots & \ddots &\vdots\\
d_{i+3,j-1} & \ldots & \ldots&d_{i+3,j+3}
\end{matrix}\right|=0.
\end{gather}
\end{enumerate}
\end{prop}

\begin{proof}The key formula is the Desnanot--Jacobi determinantal identity, see Appendix~\ref{DJid}. Consider the following piece of symplectic $2$-frieze, where the upper case (resp.\ lower case) entries are considered black entries (resp. white entries):
\[
\begin{array}{@{}ccccccccccccccc@{}}
&&&&D \\[2pt]
&&&C&c &H \\[2pt]
&& B&b&G&f&L\\[2pt]
&A&a&F &e&K&i&P\\[2pt]
&& E&d&J&h&O\\[2pt]
&&&I&g&N \\[2pt]
&&&&M
\end{array}
\]
By Desnanot--Jacobi identity one has the following relation between the minors
\begin{equation}\label{33}
\left|
\begin{matrix}
 A & B & C \\
 E & F & G \\
 I & J & K
\end{matrix}
\right|
F=
\left|
\begin{matrix}
 A & B \\
 E & F
\end{matrix}
\right|
\left|
\begin{matrix}
F & G \\
J & K
\end{matrix}
\right|
-
\left|
\begin{matrix}
 E & F \\
 I & J
\end{matrix}
\right|
\left|
\begin{matrix}
 B & C \\
 F & G
\end{matrix}
\right|.
\end{equation}

Using the frieze rule in the above equation one obtains the relation
\[
\left|
\begin{matrix}
 A & B & C \\
 E & F & G \\
 I & J & K
\end{matrix}
\right|
F=
ae-bd=F^{2},
\]
from which we deduce \eqref{C1} by cancelling out $F\not=0$:
\[
\left|
\begin{matrix}
 A & B & C \\
 E & F & G \\
 I & J & K
\end{matrix}
\right|
=F.
\]
Using again Desnanot--Jacobi identity, and then \eqref{C1} one obtains
\begin{gather*}
\left|
\begin{matrix}
 A & B & C &D \\
 E & F & G &H\\
 I & J & K &L\\
 M&N&O&P
\end{matrix}
\right|
\cdot
\left|
\begin{matrix}
F & G \\
J & K
\end{matrix}
\right|
=
\left|
\begin{matrix}
 A & B & C \\
 E & F & G \\
 I & J & K
\end{matrix}
\right|
\cdot
\left|
\begin{matrix}
 F & G &H\\
J & K &L\\
N&O&P
\end{matrix}
\right|
\\
\hphantom{\left|
\begin{matrix}
 A & B & C &D \\
 E & F & G &H\\
 I & J & K &L\\
 M&N&O&P
\end{matrix}
\right|
\cdot
\left|
\begin{matrix}
F & G \\
J & K
\end{matrix}
\right|
=}{}
-\left|
\begin{matrix}
 E & F & G \\
 I & J & K \\
 M&N&O
\end{matrix}
\right|
\cdot
\left|
\begin{matrix}
 B & C &D \\
 F & G &H\\
J & K &L
\end{matrix}
\right| =
FK-JG,
\end{gather*}
from which we deduce \eqref{C2} by cancelling out $FK-JG=e\not=0$:
\begin{gather*}
\left|
\begin{matrix}
 A & B & C &D \\
 E & F & G &H\\
 I & J & K &L\\
 M&N&O&P
\end{matrix}
\right| =1.
\end{gather*}
Applying again Desnanot--Jacobi identity to compute the $(5\times5)$-minors and using~\eqref{C1} and~\eqref{C2} already established one gets that all $(5\times5)$-minors vanish. Hence~\eqref{C3} holds.
\end{proof}

\begin{defn}\label{tame}A symplectic $2$-frieze is called \textit{tame} if the three conditions \eqref{C1}, \eqref{C2} and \eqref{C3} are all satisfied.
\end{defn}

By Proposition~\ref{3344} the symplectic $2$-friezes containing no zero entries are all tame. However the conditions of tame friezes allow to have zero entries, see Examples~\ref{extame1} and~\ref{extame2} below.

\begin{rem}\label{remtame}Note that generically the three conditions for the tameness property of a~symplectic 2-frieze are automatically satisfied according to the arguments given in the proof of Proposition~\ref{3344}. More precisely,
\begin{itemize}\itemsep=0pt
\item[--] \eqref{C1} is satisfied for a given $(i,j)$ whenever $d_{i,j}\not=0$;
\item[--] when \eqref{C1} is satisfied, then~\eqref{C2} is automatically satisfied for a given $(i,j)$ whenever $d_{i+\half,j+\half}\not=0$;
\item[--] when \eqref{C1} and \eqref{C2} are satisfied, then \eqref{C3} is automatically satisfied for a given $(i,j)$ whenever $d_{i+1,j+1}\not=0$.
\end{itemize}
\end{rem}

\begin{ex}\label{extame1}
(a) In the following array of width one, containing a black zero entry, the frieze rule is satisfied for any values of $x$:
\[
\includegraphics{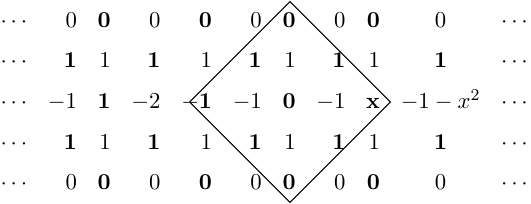}
\]

The next two entries on the right can be computed uniquely whenever $x\big(1+x^{2}\big)\not=0$. But if we want the array to be tame, the condition \eqref{C1} centered at the black~0 imposes{\samepage
\[
\left|
\begin{matrix}
 -1& 1 & 0 \\
 1 & 0 & 1 \\
 0 & 1 & x
\end{matrix}
\right|=0,
\]
so that it determines $x=1$. This will lead to the (6-periodic) tame frieze~\eqref{fw10}.}

(b) In the following array, containing a white zero entry, the frieze rule is satisfied for any values of $x$ (that will then uniquely determine $y$ and $z$):
\begin{gather*}
\begin{array}{@{}ccccccccccccccccccccccccc@{}}
\cdots&1&&\mathbf{1}&&1&\mathbf{1}&1&&\mathbf{1}&&1&\cdots\\[1pt]
\cdots&\mathbf{\sqrt2}&&1&&\mathbf{\sqrt2}&0&\mathbf{-\sqrt2+1}&&x&&\mathbf{y}&\cdots\\[1pt]
\cdots&-1&&\mathbf{1}&&-2&\mathbf{-2+\sqrt2}&-3+2\sqrt2&&\mathbf{1}&&z&\cdots\\[1pt]
\cdots&\mathbf{1}&&1&&\mathbf{1}&1&\mathbf{1}&&1&&\mathbf{1}&\cdots
\end{array}\end{gather*}
\end{ex}
The tameness condition \eqref{C1} is also satisfied for any values of $x$, since there are no black zeroes. Condition \eqref{C2}, centered at the white 0, imposes
\[
\left|
\begin{matrix}
\sqrt2& 1 & 0 &0 \\
 1 & \sqrt2 & 1&0 \\
 1 & -2+\sqrt2 & -\sqrt2+1&1\\
 0&1&1&y
\end{matrix}
\right|=1,
\]
so that one solves $y=\sqrt2-1$, and then $x=-4+2\sqrt2$, $z=1$. Then there is a unique way to extend the array to get a tame symplectic 2-frieze, which will be 7-periodic.

\begin{ex}\label{extame2}
(a) The following symplectic $2$-frieze
\begin{gather*}
\begin{array}{@{}ccccccccccccccccccc@{}}
\cdots&1&\mathbf{1}&1&\mathbf{1}&1&\mathbf{1}&1&\mathbf{1}&1&\mathbf{1}&1&\cdots\\
\cdots&\mathbf{0}&0&\mathbf{0}&0&\mathbf{0}&0&\mathbf{0}&0&\mathbf{0}&0&\mathbf{0}&\cdots\\
\cdots &0&\mathbf{0}&0&\mathbf{0}&0&\mathbf{0}&0&\mathbf{0}&0&\mathbf{0}&0&\cdots\\
\cdots&\mathbf{1}&1&\mathbf{1}&1&\mathbf{1}&1&\mathbf{1}&1&\mathbf{1}&1&\mathbf{1}&\cdots
\end{array}\end{gather*}
is not tame because condition \eqref{C1}, centered at a black 0 of the second row, fails
\[
\left|
\begin{matrix}
 0& 0 & 1 \\
 1 & 0 & 0 \\
 0 & 1 & 0
\end{matrix}
\right|\not=0.
\]

(b) The following symplectic $2$-frieze
\begin{gather*}
\begin{array}{@{}ccccccccccccccccccc@{}}
\cdots&1&\mathbf{1}&1&\mathbf{1}&1&\mathbf{1}&1&\mathbf{1}&1&\mathbf{1}&1&\cdots\\
\cdots&\mathbf{0}&0&\mathbf{0}&0&\mathbf{0}&0&\mathbf{0}&0&\mathbf{0}&0&\mathbf{0}&\cdots\\
\cdots &0&\mathbf{0}&0&\mathbf{0}&0&\mathbf{0}&0&\mathbf{0}&0&\mathbf{0}&0&\cdots\\
\cdots&\mathbf{0}&0&\mathbf{0}&0&\mathbf{0}&0&\mathbf{0}&0&\mathbf{0}&0&\mathbf{0}&\cdots\\
\cdots&1&\mathbf{1}&1&\mathbf{1}&1&\mathbf{1}&1&\mathbf{1}&1&\mathbf{1}&1&\cdots
\end{array}\end{gather*}
is not tame because condition~\eqref{C2}, centered at a white 0 of the third row, fails
\[
\left|
\begin{matrix}
0& 0& 0 & 1 \\
 1 & 0 & 0 &0 \\
 0 & 1 & 0 &0\\
 0&0&1&0
\end{matrix}
\right|\not=1.
\]

(c) There is only one tame symplectic $2$-frieze with constant first row equal to 0. It is the following frieze of width 7:
\begin{gather}\label{sing7}
\begin{array}{@{}rrrrrrrrrrrrrrrrrrrrrr@{}}
\cdots&1&\mathbf{1}&&1&\mathbf{1}&&1&\mathbf{1}&&1&\mathbf{1}&&1&\mathbf{1}&\cdots\\
\cdots&\mathbf{0}&0&&\mathbf{0}&0&&\mathbf{0}&0&&\mathbf{0}&0&&\mathbf{0}&0&\cdots\\
\cdots &0&\mathbf{0}&&0&\mathbf{0}&&0&\mathbf{0}&&0&\mathbf{0}&&0&\mathbf{0}&\cdots\\
\cdots&\mathbf{0}&0&&\mathbf{0}&0&&\mathbf{0}&0&&\mathbf{0}&0&&\mathbf{0}&0&\cdots\\
\cdots&1&\mathbf{-1}&&1&\mathbf{-1}&&1&\mathbf{-1}&&1&\mathbf{-1}&&1&\mathbf{-1}&\cdots \\
\cdots&\mathbf{0}&0&&\mathbf{0}&0&&\mathbf{0}&0&&\mathbf{0}&0&&\mathbf{0}&0&\cdots\\
\cdots &0&\mathbf{0}&&0&\mathbf{0}&&0&\mathbf{0}&&0&\mathbf{0}&&0&\mathbf{0}&\cdots\\
\cdots&\mathbf{0}&0&&\mathbf{0}&0&&\mathbf{0}&0&&\mathbf{0}&0&&\mathbf{0}&0&\cdots\\
\cdots&1&\mathbf{1}&&1&\mathbf{1}&&1&\mathbf{1}&&1&\mathbf{1}&&1&\mathbf{1}&\cdots
\end{array}
\end{gather}
\end{ex}

\begin{rem}The conventional choice to extend the array of a symplectic $2$-frieze with three rows of zeros at the top and bottom and then by antiperiodicty $d_{i+w+5,j}=d_{i,j+w+5}=-d_{i,j}$ (cf.\ Section~\ref{nota}) is justified as the only choice that extends a $2$-frieze with no zero entry to an infinite tame array.
\end{rem}

\subsection{SL-friezes and difference equations}
By definition, the tameness of the symplectic frieze forces the subarray of the black entries to form
a tame $\SL_{4}$-frieze. Hence, many properties of the symplectic friezes are immediate consequences of known results on $\SL$-friezes, see Appendix \ref{apSL}.

\begin{prop}\label{propfrieq} Let $w$ be the width of the friezes.
\begin{enumerate}\itemsep=0pt
\item[$(i)$] The subarray of the black entries in a tame symplectic $2$-frieze forms a tame $\SL_{4}$-frieze in which the ajacent $(3\times 3)$-minors are equal to the central element. Conversely, every such $\SL_{4}$-frieze determines a unique tame symplectic $2$-frieze.
\item[$(ii)$] The subarray of the black entries in a tame symplectic $2$-frieze forms a tame $\SL_{4}$-frieze invariant under the glide reflection along the median line:
\begin{gather*} d_{j-w-2,i-3}=d_{i,j}=d_{j+3, i+w+2}, \qquad i\leq j \leq i+w-1.\end{gather*}
Conversely, every such $SL_{4}$-frieze determines a unique tame symplectic $2$-frieze.
\item[$(iii)$] In \looseness=-1 a tame symplectic $2$-frieze the entries $a_{i}:=d_{i,i}$, $b_{i}:=d_{i-\half, i-\half}$ define the coefficients of a~$(w+5)$-superperiodic equation of the form \eqref{recur}. Moreover, the sequences on the diagonals $(d_{i_{0},i})_{i\in \Z}$, for every fixed $i_{0}\in \Z$, are solutions of the equation \eqref{recur}, i.e., for all $i_{0}\in \Z$
\begin{gather*}
d_{i_{0},i}=a_i d_{i_{0},i-1}-b_i d_{i_{0},i-2} + a_{i-1}d_{i_{0},i-3}-d_{i_{0},i-4}, \qquad \text{for all} \ i\in \Z.
\end{gather*}
Conversely, every $n$-superperiodic equation of the form \eqref{recur} determines a unique tame symplectic $2$-frieze.
\end{enumerate}
\end{prop}

\begin{proof}In a tame symplectic $2$-frieze, the subarray of the black entries forms a tame $\SL_{4}$-frieze according to conditions \eqref{C2} and \eqref{C3} required in Definition~\ref{tame} for the tame symplectic friezes. From \cite{MGOST}, we already know that this $\SL_{4}$-frieze corresponds to a superperiodic difference equation of order 4, see Proposition \ref{frieq} for $k=3$. Let us simply denote by
\begin{gather}\label{REqk3}
V_{i}=a_i V_{i-1}-b_i V_{i-2} + c_{i}V_{i-3}-V_{i-4},
\end{gather}
the corresponding superperiodic equation, where we use the simpler notation $a_{i}$, $b_{i}$, $c_{i}$ for the constant coefficients $a_{i}^{1}$, $a_{i}^{2}$, $a_{i}^{3}$ of~\eqref{REq}. These coefficients can be explicitly computed from the entries of the frieze, see~\eqref{DuDeT}.

In general $\SL_{k}$-friezes are not invariant under a glide reflection for $k\geq 3$, unlike Coxeter's friezes (i.e., $\SL_{2}$-friezes). We first show the following lemma giving equivalent conditions to the invariance under a glide reflection in the case of $\SL_{4}$-friezes.
\begin{lem}Given a tame $\SL_{4}$-frieze, the following properties are equivalent:
\begin{enumerate}\itemsep=0pt
\item[$(1)$] all $(3\times 3)$-minors of adjacent entries are equal to their central elements,
\item[$(2)$] the array is invariant under a glide reflection along the median line,
\item[$(3)$] the corresponding superperiodic equation is of the form~\eqref{recur}.
\end{enumerate}
\end{lem}

\begin{proof} We use the formulas given in Appendix \ref{apSL} (in our case $k=3$).
Formula \eqref{dijminor} gives
\begin{gather*}
d_{i,j}=\left\vert
\begin{matrix}
d_{j-w-3,i-4}&d_{j-w-3,i-3}&d_{j-w-3,i-2}\\
d_{j-w-2,i-4}&d_{j-w-2,i-3}&d_{j-w-2,i-2}\\
d_{j-w-1,i-4}&d_{j-w-1,i-3}&d_{j-w-1,i-2}
\end{matrix}
\right\vert.
\end{gather*}
The glide symmetry of item (2) in the above lemma can be expressed by $d_{i,j}=d_{j-w-2,i-3}$ for all $i$, $j$. Hence we see that (1) is equivalent to (2).

Now we compute the coefficients of \eqref{REqk3} with formula~\eqref{DuDeT} for $k=3$. We obtain $c_{i}=a_{i}^{3}$ as a $(1\times 1)$-minor and $a_{i}=a_{i}^{1}$ as a $(3\times 3)$-minor, namely
 \begin{gather*}
c_{i}=d_{i+2,i+w+1}, \qquad \text{and} \qquad
a_{i-1}=
\left|
\begin{matrix}
d_{i+1,i+w}&1&0\\
d_{i+2,i+w}&d_{i+2,i+w+1}&1\\
d_{i+3,i+w}&d_{i+3,i+w+1}&d_{i+j+1,i+w+2}
\end{matrix}
\right|.
\end{gather*}
The property in item~(3) can be expressed by $c_{i}=a_{i-1}$ for all~$i$. Hence, we see that~(1) implies~(3). Finally, if~(3) is satisfied, i.e., $c_{i}=a_{i-1}$ for all $i$, formulas~\eqref{DetEq1} and~\eqref{DetEq2} respectively give
\begin{gather*}
d_{i,j}=\left|
\begin{matrix}
a_{i}&1&\\
b_{i+1}&a_{i+1}&1&\\
a_{i+1}&\ddots&\ddots&\ddots\\
1&&&\ddots&\ddots\\
&\ddots&\ddots&&a_{j}&1\\
&&1&a_{j-1}&b_{j}&a_{j}
\end{matrix}
\right|,\\
d_{j+3,i+w+2}=
\left|
\begin{matrix}
a_{i}&b_{i+1}&a_{i+1}& 1&&\\
1&a_{i+1}&b_{i+2}&a_{i+2}& 1&\\
&1&\ddots&\ddots&& 1\\
&&\ddots&&\ddots&a_{j-1}\\
&&& 1&a_{j-1}&b_{j}\\
&&&& 1&a_{j}
\end{matrix}
\right|.
\end{gather*}
From which we deduce the glide symmetry of item (2). Hence (3) implies~(2). The lemma is proved.
\end{proof}

In a tame symplectic $2$-frieze, the tame $\SL_{4}$-frieze formed by the subarray of the black entries satisfies item (1) of the above lemma according to condition~\eqref{C1} required in Definiton \ref{tame} for the tame symplectic friezes. Therefore, the direct statements in items~(i),~(ii),~(iii) of Proposition~\ref{propfrieq} are immediate consequences of the above lemma.

Conversely, starting from a tame $\SL_{4}$-friezes $(d_{i,j})_{i,j\in \Z}$ one can complete the array by computing all the $(2\times 2)$-minors $d_{i-\half,j-\half}:= d_{i,j}d_{i-1,j-1}-d_{i-1,j}d_{i,j-1}$. If the initial $\SL_{4}$-frieze has all its $3\times 3$-minors equal to their central elements then the completed array will satisfy the $2$-frieze rule thank to the relation~\eqref{33}. This establishes all the converse statements in~(i),~(ii),~(iii) of Proposition~\ref{propfrieq}.
\end{proof}

\subsection{Determinantal formulas for the entries in a symplectic 2-frieze}
The entries of a tame symplectic $2$-frieze are given by multi-diagonals determinants.
\begin{prop}\label{detentry}
Consider a tame symplectic $2$-frieze, and
let
 $a_{i}:=d_{i,i}$, $b_{i}:=d_{i-\half, i-\half}$ be the entries in the non trivial bottom row.
 The entries of the frieze are given by
\begin{gather*}
d_{i,j}=
\left|\begin{matrix}
a_{i}& b_{i+1} & a_{i+1} &1 \\
1& a_{i+1}& b_{i+2} & a_{i+2} &1\\
 & \ddots & \ddots &\ddots &\ddots &\ddots \\
&& 1& a_{j-2}& b_{j-1} & a_{j-1} \\
 &&&1& a_{j-1}& b_{j} \\
&&&&1&a_{j}
\end{matrix}
\right|,\\
d_{i-\half,j-\half}=
\left|
\begin{matrix}
b_{i}&a_{i}&1\\
a_{i}& b_{i+1} & a_{i+1} &1 \\
1& a_{i+1}& b_{i+2} & a_{i+2} &1\\
 & \ddots & \ddots &\ddots &\ddots &\ddots \\
&& 1& a_{ij-2}& b_{j-1} & a_{j-1} \\
 &&&1& a_{j-1}& b_{j}
\end{matrix}
\right|,
\end{gather*}
for all $i,j\in\Z$.
\end{prop}

\begin{proof}We know by Proposition \ref{propfrieq}(iii) that $(a_{i}, b_{i})$ give the coefficients of the superperiodic equation associated with the $\SL_{4}$-frieze formed by the black entries of the symplectic frieze. Hence, the formula for $d_{i,j}$ is given by~\eqref{DetEq1}. Using the frieze rule $d_{i-\half,j-\half}= d_{i,j}d_{i-1,j-1}-d_{i-1,j}d_{i,j-1}$ and Desnanot--Jacobi identity in the determinant of $d_{i-1,j}$ one deduces the above formula for $d_{i-\half,j-\half}$.
\end{proof}

\subsection{Proofs of Theorems \ref{peri}, \ref{SLcar} and \ref{recuriso}}\label{secproofth123}
The theorems are mainly consequences and reformulations of Proposition \ref{propfrieq}.
Theorem \ref{peri} is an immediate consequence of Proposition \ref{propfrieq}(ii).
Theorem \ref{SLcar}(1) and (2) are also immediate consequences of Proposition \ref{propfrieq}(i) and~(ii). Theo\-rem~\ref{SLcar}(3) is obtained by applying the Gale duality, see Appendix \ref{friezeG}. More precisely, if~$a_{i}$,~$b_{i}$ are the coefficients of the associated difference equation (i.e., the entries in the first row of the symplectic frieze) then the corresponding $\SL_{w+1}$-frieze of width 3 is given as the following array
\[
\begin{array}{@{}cccccccccccc@{}}
& 1 & & 1 & & 1 & & 1 & & 1 & & \\
\cdots &&a_{0}& & a_{1} && a_{2} && a_{3} &&\cdots \\
 &b_{0}& & b_{1} && b_{2} && b_{3} &&\cdots \\
 \cdots &&a_{0}& & a_{1} && a_{2} && a_{3} && \cdots\\
 & 1 & & 1 & & 1 & & 1 & & 1 & &
\end{array}
\]

Theorem \ref{recuriso} is an immediate consequence of Proposition~\ref{propfrieq}(iii).

\section{Moduli space of Legendrian configurations}\label{sympgeo}

\subsection[Lift of the projective $n$-gons in $\C^{4}$]{Lift of the projective $\boldsymbol{n}$-gons in $\boldsymbol{\C^{4}}$}

Consider the symplectic space $\big(\C^{4}, \om_{0}\big)$ where $\om_{0}$ is the symplectic form given by the matrix
\[
\Omega_{0}=
\left(
\begin{matrix}
 \hphantom{-}0&\hphantom{-}0 & 1& 0\\
 \hphantom{-}0&\hphantom{-}0 &0 &1\\
- 1 & \hphantom{-}0 &0&0 \\
 \hphantom{-}0&-1&0&0
\end{matrix}
\right).
\]
 This naturally equips the projective space $\CP^{3}$ with a standard contact structure $(H_{v})_{v\in \CP^{3}}$, where $H_{v}$ is the projectivization of the hyperplane $v^{\bot}$ orthogonal to the line $v$ in $\big(\C^{4}, \om_{0}\big)$.

A Legendrian $n$-gon in $\CP^{3}$, see Definition \ref{defngon}, can be lifted to $\big(\C^{4}, \om_{0}\big)$ in a unique way (up to a sign) by imposing some normalization conditions. The normalized lift provides us with a~system of~$2n$ parameters for the Legendrian $n$-gons. The general statement in higher dimension can be found in~\cite{CoOv}. We specify the construction in the particular case of dimension~4 and sketch the proof in this case.

\begin{prop}[\cite{CoOv}]\label{proplift} Let $n$ be an odd integer. Let $v=(v_{i})_{i\in \Z}$ be a generic Legendrian $n$-gon in $\big(\CP^{3}, \om_{0}\big)$.
 \begin{enumerate}\itemsep=0pt
\item[$1.$] There exists a lift $V=(V_{i})_{i\in\Z}$ in $\C^{4}$ of $v$ satisfying the following ``normalization conditions''
\begin{enumerate}\itemsep=0pt
\item[$(a)$] $V_{i+n}=-V_{i}$, for all $i$,
\item[$(b)$] $\om_{0}(V_{i}, V_{i+1})=0$ for all $i$,
\item[$(c)$] $\om_{0}(V_{i}, V_{i+2})=1$ for all $i$.
\end{enumerate}
\item[$2.$] The sequences $V$ and $-V$ are the unique lifts of $v$ satisfying the above conditions.
\item[$3.$] The sequence $V$ satisfies
\begin{gather}\label{eqlift}
V_{i}=a_i V_{i-1}-b_i V_{i-2} + a_{i-1}V_{i-3}-V_{i-4},
\end{gather}
for all $i\in \Z$ for some unique $n$-periodic sequences of complex numbers $a(v)=(a_{i})_{i\in \Z}$ and $b(v)=(b_{i})_{i\in \Z}$.
\item[$4.$] A Legendrian $n$-gon $v'$ satisfies $a(v')=a(v)$ and $b(v')=b(v)$ if and only if~$v'$ and $v$ are $\mathrm{P}\Sp_{4}$-equivalent.
\end{enumerate}
 \end{prop}

\begin{proof}
We start with an arbitrary lift $(V_{1}, \ldots, V_{n})$ of $(v_{1}, \ldots, v_{n})$.
The sequence automatically satisfies $\om_{0}(V_{i}, V_{i+1})=0$ since by definition the Legendrian $n$-gons satisfy $v_{i+1}\in H_{v_{i}}$.

Generically $v_{i+2}\not\in H_{v_{i}}$ so we get $n$ non zero constants $\gamma_{i}:=\om_{0}(V_{i}, V_{i+2})$, $1\leq i \leq n-2$, $\gamma_{n-1}:=\om_{0}(V_{n-1}, V_{1})$, $\gamma_{n}:=\om_{0}(V_{n}, V_{2})$. When $n$ is odd, the following system of $n$ equations
\begin{gather}
 \om_{0}(\lambda_{1}V_{1}, \lambda_{3}V_{3}) = \lambda_{1}\lambda_{3} \gamma_{1} = 1, \nonumber \\
 \cdots \cdots\cdots\cdots\cdots\cdots\cdots\cdots\cdots\cdots\cdots\cdots\cdots \nonumber \\
 \om_{0}(\lambda_{n-2}V_{n-2}, \lambda_{n}V_{n}) = \lambda_{n-2}\lambda_{n} \gamma_{n-2}=1, \nonumber \\
 \om_{0}(\lambda_{n-1}V_{n-1}, \lambda_{1}V_{1}) = \lambda_{n-1}\lambda_{1} \gamma_{n-1}=-1, \nonumber \\
 \om_{0}(\lambda_{n}V_{n}, \lambda_{2}V_{2}) = \lambda_{n}\lambda_{2} \gamma_{n}=-1\label{syst}
\end{gather}
has exactly two opposite solutions $\pm(\lambda_{1}, \ldots, \lambda_{n})$ in terms of $\gamma_{i}$, $1\leq i \leq n$. We rescale the~$V_{i}$'s to~$\lambda_{i}V_{i}$ by using one of the solutions of the system. The rescaled sequence $(V_{i})$ satisfies the condition (c). We extend the sequence by antiperiodicity, i.e., we set $V_{i+kn}=(-1)^{k}V_{i}$, for all $1\leq i \leq n$, $k\in \Z$. The sequence $V$ and its opposite $-V$ are the only ones satisfying the conditions of~(1) by unicity, up to a sign, of the solution of the system~\eqref{syst}.

Items (1) and (2) are proved.

Every four consecutive points of the normalized sequence $V$ form a basis of $\C^{4}$, so that there is a unique sequence of coefficients $(a_{i},b_{i},c_{i},d_{i})_{i}$, which is $n$-periodic and satisfies
\begin{equation}\label{eqcoef}
V_{i}=a_i \,V_{i-1}-b_i\,V_{i-2} + c_{i}V_{i-3}-d_{i}V_{i-4},
\end{equation}
for all $i\in \Z$.

Applying $\om_{0}(V_{i-2}, -)$ in \eqref{eqcoef} and using the normalization conditions of~(1), we get $d_{i}=1$ for all $i\in \Z$.
Applying $\om_{0}(V_{i-3}, -)$ in \eqref{eqcoef}, we obtain $a_{i}=\om_{0}(V_{i-3}, V_{i})$ for all $i\in \Z$, that gives by a shift $a_{i-1}=\om_{0}(V_{i-4}, V_{i-1})$. And finally applying $\om_{0}(-,V_{i-1})$ in \eqref{eqcoef}, we obtain $c_{i}=\om_{0}(V_{i-4}, V_{i-1})=a_{i-1}$. Item~(3) is proved.

 Let $v$ and $v'$ be two Legendrian $n$-gons such that their normalized lifts $V$ and $V'$ satisfy the same recurrence relation \eqref{eqlift}. Define $T$ in $\GL_{4}$ such that $TV_{i}=V'_{i}$ for $i=1,2,3,4$. The recurrence relation implies $TV_{i}=V'_{i}$ for all $i$. From the normalization conditions of~(1) and the fact that
 $a_{4}=\om_{0}(V_{1}, V_{4})=\om_{0}(V'_{1}, V'_{4})$, one can see that the matrix $T$ preserves $\om_{0}$. Hence, $T$ is an element of $\Sp_{4}$ that transforms $v$ to $v'$.

 Conversely if $v'$ and $v$ are $\mathrm{P}\Sp_{4}$-equivalent then there exists $T\in \Sp_{4}$ that transforms the normalized lifts $V$ to $V'$. Therefore the sequences $V$ and $V'$ satisfy the same recurrence relation. Item (4) is proved.
\end{proof}

\begin{rem}[proof of Theorem \ref{thmngonfri}] From Proposition \ref{proplift} one can identify the Legendrian $n$-gons, modulo $\mathrm{P}\Sp_{4}$-equivalence, with $n$-superperiodic equations of the form \eqref{eqlift}, provided $n$ is odd.
Therefore Theorem \ref{thmngonfri} is an immediate consequence of Theorem \ref{recuriso}.
\end{rem}

\subsection{Symplectic 2-friezes from Legendrian $n$-gons}
One can obtain a tame symplectic $2$-frieze from the lift of a Legendrian $n$-gon using the canonical symplectic form $\omega_{0}$ on $\C^{4}$.
\begin{prop}Let $v=(v_{i})$ be a Legendrian $n$-gon in $\CP^{3}$ and $V=(V_{i})$ be its normalized lift in $\big(\C^{4}, \omega_{0}\big)$ given by Proposition~{\rm \ref{proplift}}. The following assignment
\begin{gather*}
d_{i,j}:=\omega_{0}(V_{i-3}, V_{j}), \qquad \text{for all} \ i,j \in \Z
\end{gather*}
defines the black array of a tame symplectic $2$-frieze of width $w=n-5$.
\end{prop}

\begin{proof}Using the normalization properties of $V=(V_{i})$ and the linearity of $\omega_{0}$ one deduces that the sequences
\begin{gather*}
d_{i,\bullet}=(0, \ldots, 0,1, d_{i,i}, d_{i,i+1}, \ldots, d_{i,i+w-1},1)
\end{gather*}
are solutions of the same $(w+5)$-superperiodic equation of the form \eqref{eqlift} with different initial conditions
\begin{gather*}
(V_{i-4},\ldots,V_{i-1})=(0,\ldots,0,1).
\end{gather*}
By Proposition \ref{frieq} this implies that $(d_{i,j})$ form a tame $\SL_{4}$-frieze. And since the coefficients of the equation are symmetric it implies by Proposition~\ref{propfrieq} that this $\SL_{4}$-frieze is indeed the black array of a tame symplectic $2$-frieze.
\end{proof}

\begin{rem}Alternatively, one can compute the symplectic $2$-frieze using $4\times 4$-determinants involving four vertices of the $n$-gon:
\begin{gather*}
d_{i,j}=\det(V_{i-4},V_{i-3},V_{i-2}, V_{j}), \qquad d_{i-\half,j-\half}=\det(V_{i-4},V_{i-3},V_{j-1}, V_{j}).
\end{gather*}
This is obtained as a particular case of the general correspondence between $n$-gons and $\SL_{k}$-friezes without considering the extra symplectic structure, see, e.g., formula~\eqref{formul}.
\end{rem}

\subsection{Symplectic forms and 2-friezes}\label{forms}
Consider the following family of matrices of symplectic forms
\begin{gather*}
 \Omega_{a}=\left(
\begin{matrix}
 0&0 & 1& a\\
0&0 &0 &1\\
-1 &0 &0&0 \\
-a&-1&0&0
\end{matrix}\right), \qquad
 \Check{\Omega}_{a}=\left(
\begin{matrix}
 0&0 & 1& 0\\
0&0 &-a &1\\
-1 &a &0&0 \\
0&-1&0&0
\end{matrix}\right)= {}^{t}\Omega_{a}^{-1}=-\Omega_{a}^{-1},
\end{gather*}
where $a$ is a complex parameter.

\begin{prop}\label{propsympmat} In a symplectic $2$-frieze, every $4\times4$ submatrices of black entries
\begin{gather}\label{matD}
D_{i,j}:=
\left(
\begin{matrix}
 d_{i,j-3} & \ldots &d_{i,j} \\
\vdots& &\vdots \\
d_{i+3,j-3} & \ldots&d_{i+3,j+3}\\
\end{matrix}
\right)
\end{gather}
satisfies
\begin{gather*}
^{t}\!D_{i,j} \Check{\Omega}_{a_{i}}D_{i,j}= {\Omega}_{a_{j}},
\end{gather*}
where $a_{i}=d_{i,i}$ are the black entries in the first (and last) row of the $2$-frieze.
\end{prop}

\begin{proof}It is easy to check that the property holds for the matrices $D_{i,i}$ since $D_{i,i}=\Omega_{a_{i}}$. By the recurrence relation in the frieze one has $D_{i,j+1}=D_{i,j}E_{j+1}$, where $E_{j+1}$ is the matrix
\begin{gather}\label{elemat}
E_{j+1}:=\left(\begin{matrix}
 0&0 & 0& -1\\
1&0 &0 &a_{j}\\
0& 1 &0&-b_{j} \\
 0&0&1&a_{j}
\end{matrix}\right).
\end{gather}
Hence, the property for all $D_{i,j}$ follows by easy induction on~$j$.
\end{proof}

\subsection{Proof of Proposition \ref{eqvar}}\label{secproofprop}
Let $M$ be a monodromy matrix associated to the equation~\eqref{recur}. It can be defined as the product $M=E_{1}E_{2}\cdots E_{n}$, where $E_{i}$ are the matrices~\eqref{elemat}. The condition of superperiodicity of \eqref{recur} translates as $M=-\Id$. We want to prove that the algebraic conditions on the coefficients~$(a_{i})$,~$(b_{i})$ given by $M=-\Id$ are equivalent to the one given in system~\eqref{systvar}.

We apply the matrix $M$ on the $4\times 4$ black subarray of the symplectic $2$-frieze written as the matrix~\eqref{matD} for $i=j=0$, i.e., on the matrix
\begin{gather*}
D_{0,0}=\left(
\begin{matrix}
 0&0 & 1& a_{0}\\
0&0 &0 &1\\
-1 &0 &0&0 \\
-a_{0}&-1&0&0
\end{matrix}\right).
\end{gather*}
The recurrence relations in the frieze imply that $D_{0,0}M=D_{0,n}$. Hence, $M=-\Id$ is equivalent to $D_{0,n}=-D_{0,0}$. This gives 16 equations when identifying the entries of the matrices. However, since~$D_{0,n}$ satisfies the relation ${}^{t}\!D_{0,n} \Check{\Omega}_{a_{0}}D_{0,n}= {\Omega}_{a_{0}}$, according to Proposition~\ref{propsympmat}, it is enough to identify the ten entries in the triangular lower parts of the matrices, i.e.,
\begin{gather*}
\left(
\begin{matrix}
 d_{0,n-3} & \cdot &\cdot &\cdot \\
 d_{1,n-3}& d_{1,n-2}&\cdot &\cdot \\
 d_{2,n-3}& d_{2,n-2}& d_{2,n-1} &\cdot \\
 d_{3,n-3} & d_{3,n-2} &d_{3,n-2}&d_{3,n}
\end{matrix}
\right)
=
\left(
\begin{matrix}
 0&\cdot & \cdot& \cdot\\
0&0 &\cdot &\cdot\\
1 &0 &0&\cdot \\
a_{0}&1&0&0
\end{matrix}\right),
\end{gather*}
which lead to the system \eqref{systvar} when using the determinantal expressions for the entries~$d_{i,j}$ given in Proposition~\ref{detentry}.

\subsection{More correspondences}
We explain different ways to obtain symplectic $2$-friezes, Legendrian $n$-gons, and difference equations one from another.

If one knows the $(a_{i},b_{i})$-parameters of the Legendrian $n$-gon as defined in Proposition \ref{proplift}, or equivalently the $(a_{i},b_{i})$-coefficients of the associated difference equation \eqref{recur}, then the corresponding $2$-frieze can be computed in two alternative easy ways.
\begin{enumerate}\itemsep=0pt
\item One can write down the coefficients in the first non-trivial row $(\ldots, b_{i}, \textbf{a}_{i}, b_{i+1}, \textbf{a}_{i+1},\ldots)$ of the array, and compute the rest of the entries by applying the frieze rule.
\item One can recursively compute the diagonals of the black entries of the $2$-frieze using the equations~\eqref{recur} and initial conditions
$(0, \ldots, 0,1)=(V_{i-4}, \ldots, V_{i-1})=(d_{i,i-4}, \ldots, d_{i,i-1})$.
\end{enumerate}

The first procedure is limited to the case of generic elements, while the second one can be applied in all cases.

Now, we describe the correspondence between symplectic $2$-friezes and Legendrian $n$-gons without using explicitly the coefficients of the recurrence relations. For the rest of the section, we assume $n$ is odd.

Starting from a Legendrian $n$-gon $v=(v_{i})_{i\in \Z}$, we first consider the normalized lift $V=(V_{i})$ given by Proposition \ref{proplift}. We transform $(V_{i})$ to $(W_{i})$ where $W_{i}=TV_{i}$ for the matrix $T$ of $\Sp_{4}$ defined by
\begin{gather*}
T(V_{-3})= \left(
\begin{matrix}
0\\
0 \\
-1 \\
-a_{0}
\end{matrix}
\right) ,\qquad T(V_{-2})=\left(
\begin{matrix}
0\\
0 \\
0 \\
-1
\end{matrix}
\right) ,\qquad T(V_{-1})= \left(
\begin{matrix}
1\\
0 \\
0 \\
0
\end{matrix}
\right) ,\qquad T(V_{0})=
\left(
\begin{matrix}
a_{0}\\
1\\
0 \\
0
\end{matrix}
\right)
\end{gather*}
with $a_{0}=\om_{0}(V_{-3}, V_{0})$.

The sequence $W$ will have the following form
\begin{gather}\label{block}
\begin{blockarray}{@{}ccccccccccccccc@{}}
W_{-3} & W_{-2}&W_{-1} & W_{0} & W_{1}&W_{2} & W_{3} & \cdots&\cdots &W_{n-3}&W_{n-2}
&W_{n-1}&W_{n} \\[2pt]
\begin{block}{(@{}cccccccccccccc)c@{}}
0&0&1&a_{0}&\bullet&\cdots&\bullet&1&0&0&0& -1&-a_{0}\\[2pt]
0 &0 &0&1&a_{1}&\bullet&\cdots&\bullet&1&0&0&0&-1\\[2pt]
-1&0 & 0 &0&1 &a_{2}&\bullet&\cdots&\bullet&1&0&0&0\\[2pt]
-a_{0}&-1&0&0&0&1&a_{3}&\bullet&\cdots& \bullet&1&0&0\\[2pt]
\end{block}
\end{blockarray}\!\!\!\!
\end{gather}
and can uniquely be extended in a tame $\SL_{4}$-frieze. Adding the $2\times 2$ minors inside the $\SL_{4}$-frieze will lead to a symplectic $2$-frieze.

Note that the sequence $W$ does not satisfy the normalization given in Proposition~\ref{proplift}(1) for~$\om_{0}$ but for the form given by~$\Omega_{a_{0}}$ of Section~\ref{forms}

Conversely, if one cuts a $4\times n$ block of the form~\eqref{block} in the black subarray of a symplectic $2$-frieze it gives $n$ points in $\C^{4}$ satisfying the condition of Proposition~\ref{proplift}(1) for some symplectic form~$\Omega_{a}$. After renormalizing the points they will project to a Legendrian $n$-gons in $\big(\CP^{3}, \om_{0}\big)$.

\subsection[The case when $n$ is even]{The case when $\boldsymbol{n}$ is even}
In the case when $n$ is even the system of equations~\eqref{syst} splits into two subsystems: one for the $\lambda_{i}$ with $i$ odd and one for the $\lambda_{i}$ with~$i$ even.
If $n$ is a multiple of 4 the set of solutions for the subsystems is in general empty, otherwise each subsystem has two opposite sets of solutions $\pm(\lambda_{1}, \lambda_{3}, \ldots, \lambda_{n-1})$ and $\pm(\lambda_{2}, \lambda_{4}, \ldots, \lambda_{n})$. So that we get four sequences of lifted points $\pm (V_{1}, \pm V_{2}, \ldots, V_{2i-1}, \pm V_{2i}, \ldots)$ satisfying the normalisation conditions of Proposition \ref{proplift}, which define the same parameters $(a_{i},b_{i})$ up to a sign for the coefficients~$a_{i}$. This will lead to two distinct $2$-friezes that differ under the action by conjugation of the matrix
\begin{gather*}
\left(
\begin{matrix}
 -1&0 & \hphantom{-}0& 0\\
\hphantom{-}0&1 &\hphantom{-}0 &0\\
\hphantom{-}0 & 0 &-1&0 \\
\hphantom{-}0&0&\hphantom{-}0&1
\end{matrix}\right),
\end{gather*}
 on the black subarray.

 For instance, the $2$-frieze of width 3 given in \eqref{friw3} and the following one
 \begin{gather*}
\begin{array}{@{}c@{}ccccccccccccccccccccc@{}c@{}}
\cdots&1&\mathbf{1}&1&\mathbf{1}&1&\mathbf{1}&1&\mathbf{1}&1&\mathbf{1}&1&\mathbf{1}&1&\mathbf{1}& 1&\mathbf{1}&1&\mathbf{1}&\cdots\\
\cdots&\mathbf{-1}&1&\mathbf{-2}&5&\mathbf{-4}&6&\mathbf{-4}&6&\mathbf{-3}&2&\mathbf{-1}&1&\mathbf{-4}&30&\mathbf{-10 } & 4&\mathbf{-1}&1&\cdots\\
\cdots&3&\mathbf{1}&1&\mathbf{3}&14&\mathbf{10}&20&\mathbf{6 } &3&\mathbf{1}&1&\mathbf{3}&14&\mathbf{10}&20&\mathbf{6 } &3&\mathbf{1}&\cdots\\
\cdots&\mathbf{-3}&2&\mathbf{-1}&1&\mathbf{-4}&30&\mathbf{-10 } & 4&\mathbf{-1}&1&\mathbf{-2}&5&\mathbf{-4}&6&\mathbf{-4}&6&\mathbf{-3}&2&\cdots\\
\cdots&\mathbf{1}&1&\mathbf{1}&1&\mathbf{1}&1&\mathbf{1}&1&\mathbf{1}&1&\mathbf{1}&1&\mathbf{1}&1&\mathbf{1}&\mathbf{1}&1&\mathbf{1}&\cdots
\end{array}\end{gather*}
correspond to the same class of symplectic $8$-gons.

\subsection[Legendrian $n$-gons and self-dual polygons]{Legendrian $\boldsymbol{n}$-gons and self-dual polygons}
Let $v=(v_{i})$ be a Legendrian $n$-gon in $\CP^{3}$, for some odd integer $n$. The symmetry of the coefficients in the corresponding difference equation~\eqref{eqlift} (or equivalently the symmetry in the corresponding $\SL_{4}$-friezes) implies that the polygon $v$ can be identified with its dual polygon $v^{*}$ in $\big(\CP^{3}\big)^{*}$. More precisely there is a projective map that sends $v_{i}$ to $v^{*}_{i+n-2}$. The shift can be computed in terms of $\SL_{4}$-frieze using the formula~\eqref{dijminor} and the glide symmetry.

Legendrian $n$-gons are $(n-2)$-self-dual polygons, in the terminology of~\cite{FuTa}. The space of Legendrian $n$-gons is of dimension $2(n-5)$ and provides a $3$-dimensional analog of the space~$\mathcal{M}_{m,n}$ of~\cite{FuTa}.

\section{Cluster algebras}\label{cluster}

\subsection{Preliminaries}
Cluster algebras are due to Fomin and Zelevinsky \cite{FZ1,FZ2}. They are commutative associative algebras defined by generators and relations inside a field of fractions. The generators and relations of a cluster algebra are not given from the beginning but are produced recursively.

A \textit{seed} in a cluster algebra is a couple $\Sigma=((u_{1},\ldots, u_{m}), B)$, where $(u_{1},\ldots, u_{m})$ is a transcendence basis of $\C(x_{1}, \ldots, x_{m})$ over $\C$ and $B$ is a skew-symmetrizable matrix of integer entries (alternatively the matrices can be replaced by quivers).
The initial seed $\Sigma_{0}=((x_{1}, \ldots, x_{m}),B_{0})$ is the initial data needed to generate the entire cluster algebra. All seeds are created by a process of mutation from $\Sigma_{0}$. The variables created by mutations are called \textit{cluster variables}, they come grouped in a $m$-tuple called \textit{cluster}. The complex cluster algebra $\A_{\Sigma_{0}}$ is the subalgebra of $\C(x_{1}, \ldots, x_{m})$ generated by all the cluster variables.

To fix notation and convention we recall briefly the main definitions related to cluster algebras but refer to the original papers \cite{FZ1,FZ2} or the introductory texts \cite{FWZ, Kel} for details. We also use the approach of~\cite{Dup} and~\cite{Kel2}.

\subsubsection{Mutations from skew-symmetrizable matrices}
The mutation $\mu_{k}$, $1\leq k \leq m$, of a seed $\Sigma=((u_{1}, \ldots, u_{m}),B)$ where $B=(b_{i,j})_{1\leq i,j\leq m}$, is a~skew-symmetrizable matrix with integer entries, is a new seed $\Sigma'=\mu_{k}(\Sigma)=((u'_{1}, \ldots, u'_{m}),B')$ where the new variables are given by
\begin{gather*}
 u'_{k} = \frac{1}{x_{k}}\bigg(\prod_{i\colon b_{ik}>0}u_{i}^{b_{ik}}+\prod_{i\colon b_{ik}<0}u_{i}^{-b_{ik}}\bigg),\\
 u'_{\ell} = u_{\ell}, \qquad \ell\not=k,
\end{gather*}
and the new matrix $\mu_{k}(B)=B'=(b'_{i,j})_{1\leq i,j\leq m}$ given by
\begin{gather*}
b'_{ij} =
\begin{cases}
-b_{ij} & \text{if $i=k$ or $j=k$}, \\
b_{ij}+b_{ik}b_{kj} & \text{if $b_{ik}>0$ and $b_{kj}>0$},\\
b_{ij}-b_{ik}b_{kj} & \text{if $b_{ik}<0$ and $b_{kj}<0$},\\
b_{ij} & \text{otherwise}.
\end{cases}
\end{gather*}

\subsubsection{Valued quivers}

The valued quiver $Q_{v}(B)$ associated to a skew-symmetrizable matrix $B=(b_{i,j})_{1\leq i,j\leq m}$ with integer entries, is defined as follows:
\begin{itemize}\itemsep=0pt
\item $\{1, \ldots, m\}$ is the set of vertices,
\item each entry $b_{i,j}>0$ gives a weighted arrow $i \to j$ of weight $(|b_{i,j}|, |b_{j,i}|)$.
\end{itemize}

Recall that in a skew-symmetrizable matrix one has $b_{i,j}>0$ if and only if $b_{j,i}\leq 0$. When there is a weighted arrow between the vertices $i$ and $j$ the weight $|b_{i,j}|$ is ``attached'' to the vertex $i$ and the weight $|b_{j,i}|$ ``attached'' to the vertex~$j$. On the graph of a valued quiver we write the weights on the top or bottom of the arrows The weighted arrow $ \xymatrix{i\ar@{->}[r]^{(a,b)}&j}$ can be also represented by $\xymatrix{j\ar@{<-}[r]^{(b,a)}&i}$.

In the case when $(|b_{i,j}|, |b_{j,i}|)=(b, b)$, we replace the weighted arrow by $b$ arrows. For consistence of formulas it may be convenient to adopt the following convention: $\xymatrix{i\ar@{->}[r]^{(0,0)}&j}$ means no arrow between $i$ and $j$, and $\xymatrix{i\ar@{->}[r]^{(-a,-b)}&j}$ means $\xymatrix{i\ar@{<-}[r]^{(a,b)}&j}$.

Our notation differs from that of \cite{Kel2}.

\begin{ex}
For example the matrix
\[
\left(
\begin{matrix}
 \hphantom{-}0& \hphantom{-}1 & \hphantom{-}0 &-1 \\
 -2&\hphantom{-}0& -1 & \hphantom{-}4 \\
 \hphantom{-}0 & \hphantom{-}1 &\hphantom{-}0& -2 \\
 \hphantom{-}1&-2&\hphantom{-}1&\hphantom{-}0
\end{matrix}
\right)
\]
gives the following valued quiver
\[\xymatrix{
&1\ar@{->}[r]^{(1,2)}
&2\ar@{<-}[d]
\\
&4\ar@{->}[r]_{(1,2)}\ar@{->}[u]\ar@{<-}[ru]^*[@]{\hbox to 0pt{\hss \footnotesize{(2,4)}\hss}}
&3
}\]
\end{ex}

\subsubsection{Mutations of valued quivers}
The mutation at vertex $k$ of a valued quiver $Q_{v}$ gives a valued quiver $\mu_{k}(Q_{v})$ defined by the following transformations:
\begin{enumerate}\itemsep=0pt
\item For all paths $i\to k \to j$ in $Q$, change the weighted arrows
\[
\xymatrix{
&k\ar@{->}[rd]^*[@]{\hbox to -10pt{\hss \footnotesize{$(c,d)$}\hss}}&\\
i\ar@{->}[rr]_{(e,f)}\ar@{->}[ru]^*[@]{\hbox to 1pt{\hss \footnotesize{$(a,b)$}\hss}}&&j
}
\qquad \text{to} \qquad
\xymatrix{
&k\ar@{->}[rd]^*[@]{\hbox to -10pt{\hss \footnotesize{$(c,d)$}\hss}}&\\
i\ar@{->}[rr]_{(e+ac, \,f+bd)}\ar@{->}[ru]^*[@]{\hbox to 1pt{\hss \footnotesize{$(a,b)$}\hss}}&&j
}
\]
and
\[
\xymatrix{
&k\ar@{->}[rd]^*[@]{\hbox to -10pt{\hss \footnotesize{$(c,d)$}\hss}}&\\
i\ar@{<-}[rr]_{(g,h)}\ar@{->}[ru]^*[@]{\hbox to 1pt{\hss \footnotesize{$(a,b)$}\hss}}&&j
}
\qquad \text{to} \qquad
\xymatrix{
&k\ar@{->}[rd]^*[@]{\hbox to -10pt{\hss \footnotesize{$(c,d)$}\hss}}&\\
i\ar@{<-}[rr]_{(g-ac,\, h-bd)}\ar@{->}[ru]^*[@]{\hbox to 1pt{\hss \footnotesize{$(a,b)$}\hss}}&&j
}
\]
where $a,b,c,d>0$ and $e,f,g,h \geq 0$.
\item Reverse the weighted arrows incident with $k$, i.e., change all
\[
\xymatrix{
i\ar@{->}[r]^{(a,b)}&k} \qquad \text{to} \qquad \xymatrix{
i\ar@{<-}[r]^{(a,b)}&j}
\]
and
\[
\xymatrix{
k\ar@{->}[r]^{(c, d)}&j} \qquad \text{to} \qquad \xymatrix{
k\ar@{<-}[r]^{(c,d)}&j}.
\]
\end{enumerate}
This rule gives the correspondence $\mu_{k}(Q_{v}(B))=Q_{v}(\mu_{k}(B))$.
Examples of mutations can be found in Examples~\ref{exF4} and~\ref{exmutfriz}.

\subsubsection{Product of valued quivers}
From the skew-symmetrizable, resp.\ skew-symmetric, companion matrices corresponding to the Dynkin type $\mC_{2}$, resp.\ $\mA_{w}$, consider the corresponding valued quivers as follows
\[
\xymatrix{
{\mC}_{2}\colon \ 1\ar@{->}[r]^{\;\;\;(1,2)}&2
&& {\mA}_{w}\colon \ 1\ar@{->}[r]&2\ar@{->}[r]&\ldots \ar@{->}[r]&w-1\ar@{->}[r]&w.}
\]
Following \cite{Kel2}, we define the square product of the valued quivers:
\begin{gather}\label{prodquiv}
\mC_{2}\square\mA_{w}:= \begin{split}&\xymatrix{
\overset{w+1}{\bullet}\ar@{->}[r]
& \overset{w+2}{\circ} \ar@{<-}[r]
&\overset{w+3}{\bullet}\ar@{->}[r]
&\cdots
&\cdots \ar@{<-}[r]
&\overset{2w-1}{\bullet}\ar@{->}[r]
& \overset{2w}{\circ} \\
\underset{1}{\circ} \ar@{<-}[r]\ar@{->}[u]^*[@]{\hbox to 10pt{\hss \footnotesize{$(1,2)$}\hss}}
&\underset{2}{\bullet}\ar@{->}[r]\ar@{<-}[u]^*[@]{\hbox to 10pt{\hss \footnotesize{$(1,2)$}\hss}}
& \underset{3}{\circ} \ar@{->}[u]^*[@]{\hbox to 10pt{\hss \footnotesize{$(1,2)$}\hss}}\ar@{<-}[r]
&\cdots&\cdots \ar@{->}[r]
& \underset{w-1}{\circ} \ar@{->}[u]^*[@]{\hbox to 10pt{\hss \footnotesize{$(1,2)$}\hss}}\ar@{<-}[r]
&\underset{w}{\bullet}\ar@{<-}[u]^*[@]{\hbox to 10pt{\hss \footnotesize{$(1,2)$}\hss}}\\
}\end{split}
\end{gather}
It is the valued quiver associated to the following $2w\times 2w$ skew-symmetrizable matrix
\begin{gather*}
B_{\mC_{2}\square\mA_{w}}=
 \left(
 \begin{array}{@{}rrrrr|rrrrrr@{}}
 0&-1&0&\;\;0&\cdots &1&0&\;\;0&0&\cdots\\
 1&0&1&0&\cdots& 0&-1&0&0&\cdots\\
 0&-1&0&-1&\cdots&0&0&1&0&\cdots\\
\vdots &&&\vdots&&\vdots &&&\vdots&\\
\hline
 -2&0&0&0&\cdots&0&1&0&0&\cdots\\
 0&2&0&0&\cdots&-1&0&-1&0&\cdots\\
 0&0&-2&0&\cdots&0&1&0&1&\cdots\\
 \vdots &&&\vdots&&\vdots &&&\vdots&
 \end{array}
 \right)
 \end{gather*}

\begin{ex}\label{exF4}
(a) For $w=1$, $\mC_{2}\square\mA_{1}$ is just the valued quiver $\xymatrix{
{\mC}_{2}:\; 1\ar@{->}[r]^{\;\;\;(1,2)}&2}$ associated to the matrix
$\left(\begin{smallmatrix}
0 & 1 \\
-2 & 0
\end{smallmatrix}
\right)$.

(b) For $w=2$, one can show that $\mC_{2}\square\mA_{2}$ is mutation equivalent to a valued quiver of Dynkin type $\mF_{4}$. Indeed,
\[\xymatrix{
\overset{3}{\bullet}\ar@{->}[r]
& \overset{4}{\circ} &
\overset{\mu_{1}}{\longrightarrow}&\overset{3}{\bullet}\ar@{->}[r]\ar@{<-}[dr]^*[@]{\hbox to -10pt{\hss \!\!\footnotesize{$(2,1)$}\hss}}
& \overset{4}{\circ} &
\overset{\mu_{3}}{\longrightarrow}&\overset{3}{\bullet}\ar@{<-}[r]\ar@{->}[dr]^*[@]{\hbox to -10pt{\hss \!\!\footnotesize{$(2,1)$}\hss}}
& \overset{4}{\circ}
&\overset{\mu_{1}}{\longrightarrow}&\overset{3}{\bullet}\ar@{<-}[r]& \overset{4}{\circ} \\
\underset{1}{\circ} \ar@{<-}[r]\ar@{->}[u]^*[@]{\hbox to 10pt{\hss \footnotesize{$(1,2)$}\hss}}
&\underset{2}{\bullet}\ar@{<-}[u]_*[@]{\hbox to 10pt{\hss \footnotesize{$(1,2)$}\hss}}
&&\underset{1}{\circ} \ar@{->}[r]\ar@{<-}[u]^*[@]{\hbox to 10pt{\hss \footnotesize{$(1,2)$}\hss}}
&\underset{2}{\bullet}\ar@{<-}[u]_*[@]{\hbox to 10pt{\hss \footnotesize{$(1,2)$}\hss}}
&&\underset{1}{\circ} \ar@{<-}[r]\ar@{->}[u]^*[@]{\hbox to 10pt{\hss \footnotesize{$(1,2)$}\hss}}
&\underset{2}{\bullet}
&&\underset{1}{\circ} \ar@{->}[r]\ar@{<-}[u]^*[@]{\hbox to 10pt{\hss \footnotesize{$(1,2)$}\hss}}
&\underset{2}{\bullet}
}
\]
\end{ex}

\begin{ex}\label{exmutfriz}
Let $\Sigma=((x_{1}, x_{2}, \ldots, x_{2w}), \mC_{2}\square\mA_{w}) $ be a seed.
A mutation at vertex $i$ for $1\leq i \leq w$ in
\eqref{prodquiv} locally changes the graph around the vertices $i$ and $w+i$ into
\begin{equation*}
\label{}
\xymatrix
 @!0 @R=2cm @C=2cm
{
w+i-1\ar@{->}[r]
& w+i \ar@{<-}[r]\ar@{->}[rd]^*[@]{\hbox to -10pt{\hss \footnotesize{$\!(2,1)$}\hss}}
&w+i+1
\\
i-1 \ar@{->}[r]\ar@{->}[u]^*[@]{\hbox to 10pt{\hss \footnotesize{$(1,2)$}\hss}}\ar@{<-}[ru]^*[@]{\hbox to 1pt{\hss \footnotesize{$(1,2)$}\hss}}
&i\ar@{<-}[r]\ar@{->}[u]^*[@]{\hbox to 5pt{\hss \footnotesize{$(1,2)$}\hss}}
&i+1 \ar@{<-}[u]^*[@]{\hbox to 5pt{\hss \footnotesize{$(1,2)$}\hss}}
}
\end{equation*}
if $i$ is even (and a similar picture if $i$ is odd but with reversed arrows), and changes the variable~$x_{i}$ to
\begin{gather}\label{mutxi}
x_{i}'=\dfrac{x_{i-1}x_{i+1}+x_{w+i}^{2}}{x_{i}}.
\end{gather}

A mutation at vertex $w+i$ for $1\leq i \leq w$ in~\eqref{prodquiv} locally changes the graph around the vertices $i$ and $w+i$ into
\begin{gather*}
\xymatrix
 @!0 @R=2cm @C=2cm
{
w+i-1\ar@{<-}[r]\ar@{->}[rd]^*[@]{\hbox to -10pt{\hss \footnotesize{$\!(2,1)$}\hss}}
&w+ i \ar@{->}[r]
&w+i+1
\\
i-1 \ar@{<-}[r]\ar@{->}[u]^*[@]{\hbox to 10pt{\hss \footnotesize{$(1,2)$}\hss}}
&i\ar@{->}[r]\ar@{->}[u]^*[@]{\hbox to 5pt{\hss \footnotesize{$(1,2)$}\hss}}\ar@{<-}[ru]^*[@]{\hbox to 1pt{\hss \footnotesize{$(1,2)$}\hss}}
&i+1 \ar@{->}[u]^*[@]{\hbox to 5pt{\hss \footnotesize{$(1,2)$}\hss}}
}
\end{gather*}
if $i$ is even (and a similar picture if $i$ is odd but with reversed arrows), and changes the variab\-le~$x_{w+i}$ to
\begin{gather}\label{mutxwi}
x_{w+i}'=\dfrac{x_{w+i-1}x_{w+i+1}+x_{i}}{x_{w+i}}.
\end{gather}
\end{ex}

\subsubsection{Bipartite belt}
A \textit{bipartite} quiver is such that the vertices can be colored in two colors so that vertices of the same color are not connected by any arrow. For instance, the quiver~\eqref{prodquiv} is bipartite.

Let $\Qc$ be a bipartite quiver. Following Fomin--Zelevinsky \cite{FZ4}, consider the iterated mutations
\begin{gather*}
\mu_+=\prod_{\text{white } k} \mu_k, \qquad
\mu_-=\prod_{\text{black } k} \mu_{k}.
\end{gather*}
Note that the mutation $\mu_k$'s with $k$ of a fixed color commute with each other since the vertices are not connected.

The mutations $\mu_+$ and $\mu_-$ on $Q$ give the same quiver with reversed orientation:
\begin{gather*}
\mu_+(\Qc)=\Qc^{{\rm op}}, \qquad \mu_-\big(\Qc^{{\rm op}}\big)=\Qc.
\end{gather*}

Let $\Sigma_0=(\underline{x}, \Qc)$ be an initial seed. The \textit{bipartite belt}, see~\cite{FZ4}, of the cluster algebra $\A_{\Sigma_{0}}$ is the collection of seeds obtained from $\Sigma_0$ by applying successively $\mu_+$ or $\mu_{-}$:
\begin{gather*}
\ldots, \ \mu_{-}\mu_+\mu_-(\Sigma_0),\
\mu_+\mu_-(\Sigma_0),\
\mu_-(\Sigma_0),\
\Sigma_0,\
\mu_+(\Sigma_0),\quad \mu_-\mu_+(\Sigma_0),\
\mu_{+}\mu_-\mu_+(\Sigma_0), \ \ldots.
\end{gather*}

\begin{rem}\label{zalamo}In general, the bipartite belt is infinite. It has been proved, see \cite{Kel,Kel2}, that for a bipartite (valued) quiver constructed as a product of two Dynkin quivers $\Delta\square\Delta'$ the mutation~$\mu_{+}\mu_{-}$ has finite order, more precisely
\begin{gather*}
(\mu_{+}\mu_{-})^{2(h+h')}=(\mu_{-}\mu_{+})^{2(h+h')}=\Id,
\end{gather*}
where $h$ and $h'$ are the Coxeter numbers of the Dynkin quivers $\Delta$, $\Delta'$ respectively. This property is known as \textit{Zamolodchikov periodicity}.
\end{rem}

\subsection{Cluster variables in the symplectic 2-friezes}\label{clustvar1}
Let $F_{w}$ be the tame symplectic $2$-frieze of width $w$ of rational fractions in the free variables $(x_{1}, x_{2}, \ldots ,{x_{2w}})$ computed with the frieze rule from
two consecutive columns containing the variables $(x_{1}, \ldots , x_{w}, \mathbf{x_{w+1}}, \ldots,\mathbf{x_{2w}})$ in a zig-zag as follows.
\begin{gather}\label{Fw}
F_{w}\colon \
\begin{matrix}
\cdots&1& 1&1&1&1&1&\cdots \\
&&x_1& \mathbf{x_{w+1}}&x'_1& \mathbf{x'_{w+1}}&& \\
&& \mathbf{x_{w+2}}&x_2& \mathbf{x'_{w+2}}&x'_2& \\
&&x_3& \mathbf{x_{w+3}}&x'_3& \mathbf{x'_{w+3}}& \\
&&\vdots &\vdots &\vdots &\vdots &\\
&&\mathbf{x_{2w}}&x_w& \mathbf{x'_{2w}}&x'_w& & \\
\cdots&1& 1&1&1&1&1&\cdots
\end{matrix}
\end{gather}
The frieze \eqref{frizC2} is exactly the frieze $F_{w}$ for $w=1$.
\begin{rem}

The frieze $F_{w}$ is well defined. The operations used to compute the entries in the symplectic $2$-frieze in terms of the initial entries $(x_{1}, x_{2}, \ldots ,{x_{2w}})$ are subtraction free. So that the rational expressions that one computes for the next entries $(x'_{1}, x'_{2}, \ldots ,{x'_{2w}})$ do not vanish. This allows to compute recursively all the entries of the frieze as elements of $\C(x_{1}, \ldots, x_{2w})$. Moreover, since the entries are non zero, the frieze $F_{w}$ is tame by Proposition~\ref{3344}.
\end{rem}

\begin{prop}\label{thmAmas2f} Let $\A=\A_{\Sigma_{0}}$ be the cluster algebra generated from the initial seed $\Sigma_{0}=(\,(x_{1},x_{2},\ldots, x_{2w}), \Qc)$, where $\Qc=\mC_{2}\square\mA_{w}$ is the valued quiver~\eqref{prodquiv}. Let $F_{w}$ be the formal frieze~\eqref{Fw} containing the same initial variables $(x_{1},x_{2},\ldots, x_{2w})$.
\begin{enumerate}\itemsep=0pt
\item[$(i)$] The entries of $F_{w}$ are all cluster variables of $\A$;
\item[$(ii)$] The pairs of consecutive columns of $F_{w}$ give all the clusters of the bipartite belt of $\A$.
\end{enumerate}
\end{prop}

\begin{proof}(i) is a trivial consequence of (ii). Statement (ii) is established by direct computations. Indeed, let $(x'_{1}, x'_{w+2}, x'_{3}, \ldots)$ and $(x'_{w+1}, x'_{2}, x'_{w+3}, \ldots)$ be the third and fourth column in $F_{w}$. The entries are deduced using the frieze rule so that we have for $i$ odd, $1\leq i \leq w$,
\begin{gather*}
x_{i}'=\dfrac{x_{i-1}x_{i+1}+x_{w+i}^{2}}{x_{i}},
\end{gather*}
and $i$ even, $1\leq i \leq w$,
\begin{gather*}
x_{w+i}'=\dfrac{x_{w+i-1}x_{w+i+1}+x_{i}}{x_{w+i}}.
\end{gather*}
Comparing with the result of the mutation $\mu_{+}$, see \eqref{mutxi}, \eqref{mutxwi}, one immediately gets
\begin{gather*}
\big(x'_{1}, x_{2}, x'_{3},\dots,x'_{2w-1}, x_{2w}, \Qc^{\rm op} \big)=\mu_{+}\Sigma_{0},\\
(x'_{1}, x'_{2}, x'_{3},\dots,x'_{2w-1}, x'_{2w}, \Qc )=\mu_{-}\mu_{+}\Sigma_{0},
\end{gather*}
and so on.
\end{proof}

\begin{rem}From the cluster algebra theory Theorem~\ref{peri} becomes an easy consequence of Proposition~\ref{thmAmas2f} and the Zamolodchikov periodicity, see Remark~\ref{zalamo} (recall that the Coxeter number for $\mC_{2}$ is 4 and the one for $\mA_{w}$ is $w+1$). However, the periodicity of the 2-friezes or the periodicity of the $\SL$-friezes can be established with elementary arguments using the identification with superperiodic difference equations, see~\cite{MGOST}.
\end{rem}

\subsection{More clusters in the 2-friezes}\label{FunClo}

From $F_{w}$ one can extract more clusters of $\A$ than the ones of the bipartite belt. These clusters are given as ``double zig-zags'' in the array. We follow \cite{MGOTaif}.

We call \textit{double zig-zag} in the 2-frieze \eqref{Fw} a sequence $\zeta=(\widetilde{x}_1,\ldots, \widetilde{x}_{w},\widetilde{y}_1,\ldots, \widetilde{y}_{w})$ where $\widetilde{x}_i$'s are white entries and $\widetilde{y}_i$'s black entries that are locally in the following 6 configurations
\begin{gather}\label{conf}
\begin{split}&
\xymatrix
 @!0 @R=1cm @C=1cm
{
&\widetilde{x}_i\ar@{--}[ld]
&\widetilde{y}_i\ar@{-}[ld]
&&
\widetilde{x}_i\ar@{--}[rd]
&\widetilde{y}_i\ar@{-}[ld]
&&
\widetilde{x}_i\ar@{--}[rd]
&\widetilde{y}_i\ar@{-}[rd]
\\
\widetilde{x}_{i+1}
&
\widetilde{y}_{i+1}
&&&
\widetilde{y}_{i+1}
&
\widetilde{x}_{i+1}
&&&
\widetilde{x}_{i+1}
&
\widetilde{y}_{i+1}
\\
&\widetilde{y}_i\ar@{-}[ld]
&\widetilde{x}_i\ar@{--}[ld]
&&
\widetilde{y}_i\ar@{-}[rd]
&\widetilde{x}_i\ar@{--}[ld]
&&
\widetilde{y}_i\ar@{-}[rd]
&\widetilde{x}_i\ar@{--}[rd]
\\
\widetilde{y}_{i+1}
&
\widetilde{x}_{i+1}
&&&
\widetilde{x}_{i+1}
&
\widetilde{y}_{i+1}
&&&
\widetilde{y}_{i+1}
&
\widetilde{x}_{i+1}
}\end{split}
\end{gather}

\begin{ex}The following sequence forms a double zig-zag in a frieze of width 5:
\begin{gather}\label{DobZag}
 \begin{matrix}
\cdots & 1& 1&1&1&1&1&\cdots
 \\[4pt]
&&&\widetilde{x}_1&\widetilde{y}_1&&&
 \\[4pt]
&&&&\widetilde{x}_2&\widetilde{y}_2&&
 \\[4pt]
&& &\widetilde{x}_3&\widetilde{y}_3&&&
 \\[4pt]
 && &\widetilde{y}_4&\widetilde{x}_4&&&
 \\[4pt]
& & \widetilde{y}_5&\widetilde{x}_5&&&&
 \\[4pt]
\cdots& 1& 1&1&1&1&1&\cdots
\end{matrix}
\end{gather}
\end{ex}

To every double zig-zag $\zeta$ one associates a valued quiver $\Qc_{\zeta}$ defined over the set of vertices $\{1,2,\ldots, 2w\}$ as follows. Each of the 6 configurations of~\eqref{conf} gives respectively the following set of weighted arrows:
\begin{gather*}
\xymatrix
@!0 @R=2cm @C=2cm
{
&w+i\ar@{<-}[r]
&w+{i+1}
&w+i\ar@{->}[r]
&w+{i+1}
&w+i\ar@{<-}[r]\ar@{->}[rd]^*[@]{\hbox to -10pt{\hss \footnotesize{$\!(2,1)$}\hss}}
&w+{i+1}
\\
&i\ar@{<-}[r]\ar@{<-}[u]^*[@]{\hbox to 5pt{\hss \footnotesize{$(1,2)$}\hss}}\ar@{->}[ru]^*[@]{\hbox to 1pt{\hss \footnotesize{$(1,2)$}\hss}}
&{i+1}\ar@{<-}[u]^*[@]{\hbox to 5pt{\hss \footnotesize{$(1,2)$}\hss}}
&i\ar@{<-}[r]\ar@{->}[u]^*[@]{\hbox to 5pt{\hss \footnotesize{$(1,2)$}\hss}}
&{i+1}\ar@{<-}[u]^*[@]{\hbox to 5pt{\hss \footnotesize{$(1,2)$}\hss}}
&i\ar@{<-}[r]\ar@{->}[u]^*[@]{\hbox to 5pt{\hss \footnotesize{$(1,2)$}\hss}}
&{i+1}\ar@{->}[u]^*[@]{\hbox to 5pt{\hss \footnotesize{$(1,2)$}\hss}}
\\
&w+i\ar@{<-}[r]\ar@{->}[rd]^*[@]{\hbox to -10pt{\hss \footnotesize{$\!(2,1)$}\hss}}
&w+{i+1}
&w+i\ar@{<-}[r]
&w+{i+1}
&w+i\ar@{<-}[r]
&w+{i+1}
\\
&i\ar@{<-}[r]\ar@{->}[u]^*[@]{\hbox to 5pt{\hss \footnotesize{$(1,2)$}\hss}}
&{i+1}\ar@{->}[u]^*[@]{\hbox to 5pt{\hss \footnotesize{$(1,2)$}\hss}}
&i\ar@{->}[r]\ar@{<-}[u]^*[@]{\hbox to 5pt{\hss \footnotesize{$(1,2)$}\hss}}
&{i+1}\ar@{->}[u]^*[@]{\hbox to 5pt{\hss \footnotesize{$(1,2)$}\hss}}
&i\ar@{<-}[r]\ar@{->}[u]^*[@]{\hbox to 5pt{\hss \footnotesize{$(1,2)$}\hss}}\ar@{->}[ru]^*[@]{\hbox to 1pt{\hss \footnotesize{$(1,2)$}\hss}}
&{i+1}\ar@{<-}[u]^*[@]{\hbox to 5pt{\hss \footnotesize{$(1,2)$}\hss}}
}
\end{gather*}
when $i$ is odd, and same pictures with reversed arrows when $i$ is even.
\begin{ex} The quiver associated to the double zig-zag of \eqref{DobZag} is
\begin{gather*}
\xymatrix
 @!0 @R=2cm @C=2cm
{
6\ar@{<-}[r]\ar@{->}[rd]^*[@]{\hbox to -10pt{\hss \footnotesize{$\!(2,1)$}\hss}}
&7\ar@{->}[r]
&8\ar@{->}[r]
&9\ar@{->}[r]\ar@{<-}[rd]^*[@]{\hbox to -10pt{\hss \footnotesize{$\!(2,1)$}\hss}}
&10
\\
1 \ar@{<-}[r]\ar@{->}[u]^*[@]{\hbox to 10pt{\hss \footnotesize{$(1,2)$}\hss}}
&2\ar@{->}[r]\ar@{->}[u]^*[@]{\hbox to 5pt{\hss \footnotesize{$(1,2)$}\hss}}\ar@{<-}[ru]^*[@]{\hbox to 1pt{\hss \footnotesize{$(1,2)$}\hss}}
&3 \ar@{->}[u]^*[@]{\hbox to 5pt{\hss \footnotesize{$(1,2)$}\hss}}\ar@{<-}[r]
&4\ar@{->}[r]\ar@{<-}[u]^*[@]{\hbox to 5pt{\hss \footnotesize{$(1,2)$}\hss}}
&5\ar@{<-}[u]^*[@]{\hbox to 5pt{\hss \footnotesize{$(1,2)$}\hss}}
}
\end{gather*}
\end{ex}

One can extend Proposition \ref{thmAmas2f} with the following statement.
\begin{prop}\label{clustzig}
For every double zig-zag $\zeta$ in $F_{w}$, the couple $(\zeta, \Qc_{\zeta})$ is a seed of
$\A_{\Sigma_{0}}$.
\end{prop}

\begin{proof}\looseness=-1 The statement follows from the mutation rule and the frieze rule. Every double zig-zag can be redressed in two consecutive columns by elementary moves that correspond to mutations.
\end{proof}

\subsection{The cluster variety of tame symplectic 2-friezes}\label{clustvar2}

Let $\A=\A_{\Sigma_{0}}$ be the cluster algebra generated from the initial seed $\Sigma_{0}=((x_{1},x_{2},\ldots, x_{2w}), \allowbreak \mC_{2}\times\mA_{w})$, and let $\chi=(u_{1},\ldots, u_{2w})$ be an arbitrary cluster of~$\A$. We know from the cluster algebra theory that all the cluster variables of $\A$ can be expressed as Laurent polynomials with positive integer coefficients in the variables~$u_{i}$. Denote by $F_{w}(u_{1},\ldots, u_{2w})$ the frieze of Laurent polynomials obtained by expressing all the entries of~$F_{w}$ in terms of~$u_{i}$. When specializing the variables $u_{i}$ to non-zero complex numbers one gets a tame symplectic $2$-frieze over complex numbers. This construction can be summarized with the following statement.

\begin{cor}\label{clustparam} Every cluster $\chi=(u_{1},\ldots, u_{2w})$ of $\A$, defines an injective map
\begin{gather*} \phi_{\chi}\colon \ (z_{1}, \ldots, z_{2w})\mapsto F_{w}(u_{1},\ldots, u_{2w})_{|u_{i}=z_{i}}\end{gather*}
from $(\C^{*})^{2w}$ to the variety of tame symplectic $2$-friezes.
\end{cor}

Obviously the maps $\phi_{\chi}$ are not surjective. A natural question is the following.

\begin{open}\label{open1}Is the variety of tame symplectic $2$-friezes of width $w$ equal to the union over all clusters of the images of the maps $\phi_{\chi}$?
\end{open}

Let us mention that the similar question in the case of tame Coxeter's friezes has been addressed and answered by Cuntz and Holm~\cite{CuHo} (in that case the answer is positive when the width is not equal to 3 modulo 4, otherwise there exists an extra tame frieze).

\begin{defn}
A tame symplectic $2$-frieze is called \textit{cluster-singular} if it is not in any of the images of the maps $\phi_{\chi}$, otherwise it is called \textit{cluster-regular}
\end{defn}

For $w=1, 2$, one can give the following (partial) answers.

\begin{prop}\label{clustsing}\quad
\begin{enumerate}\itemsep=0pt
\item[$(i)$] There is only one cluster-singular symplectic $2$-frieze of width~$1$ given~by
\begin{gather}\label{sing1}
\begin{array}{@{}rrrrrrrrrcrrrr@{}}
\cdots&\mathbf{1}&1&\mathbf{1}&1&\mathbf{1}&1&\mathbf{1}&1&\mathbf{1}&\cdots\\
\cdots&-1&\mathbf{0}&-1&\mathbf{0}&-1&\mathbf{0}&-1&\mathbf{0}&-1 &\cdots\\
\cdots&\mathbf{1}&1&\mathbf{1}&1&\mathbf{1}&1&\mathbf{1}&1&\mathbf{1}&\cdots
\end{array}\end{gather}
\item[$(ii)$] There is at most one cluster-singular tame symplectic $2$-frieze of width~$2$.
\end{enumerate}
\end{prop}

\begin{proof}For $w=1$, a tame symplectic 2-frieze is entirely determined by two consecutive non-zero entries. Indeed, a tame symplectic 2-frieze of width 1 is necessarily 6-periodic and two consecutive non-zero entries allow to obtain~6 consecutive entries using the frieze rule. The cluster algebra of type $\mC_{2}$ has exactly 6 clusters given as two consecutive entries in the single non trivial row of the array~\eqref{frizC2}. There are only two ways to avoid two consecutive non zero entries in a symplectic 2-frieze of width 2, leading to \eqref{sing1} and to
\begin{gather}\label{sing11}
\begin{array}{rrrrrrrrrcrrrrrr}
\cdots&\mathbf{1}&1&\mathbf{1}&1&\mathbf{1}&1&\mathbf{1}&1&\mathbf{1}&1&\cdots\\
\cdots&0&\mathbf{i}&0&\mathbf{-i}&0&\mathbf{i}&0&\mathbf{-i}&0 &\mathbf{i}&\cdots\\
\cdots&\mathbf{1}&1&\mathbf{1}&1&\mathbf{1}&1&\mathbf{1}&1&\mathbf{1}&1&\cdots
\end{array}\end{gather}
One checks that the frieze \eqref{sing1} is tame. To do so one uses Remark \ref{remtame} and check \eqref{C1} and \eqref{C3} centered at a black zero entry. The only relevant calculations are to
\begin{gather*}
\left|
\begin{matrix}
0&1&0\\
1&0&1\\
0&1&0
\end{matrix}
\right| =0,
\qquad
\left|
\begin{matrix}
0& 1&0&0&0\\
1&0&1&0&0\\
0&1&0&1&0\\
0& 0&1&0&1\\
0& 0& 0&1&0
\end{matrix}
\right| =0,
\end{gather*}
which allow to conclude that \eqref{sing1} is tame. Similarly for \eqref{sing11} one computes
\begin{gather*}
\left|
\begin{matrix}
i& \hphantom{-}1&0&\hphantom{-}0\\
1&-i&1&\hphantom{-}0\\
0&\hphantom{-}1&i&\hphantom{-}1\\
0& \hphantom{-}0&1&-i
\end{matrix}
\right| =-1,
\end{gather*}
which shows that \eqref{sing11} is not tame since \eqref{C2} fails. Hence (i).

For $w=2$, we start by establishing two lemmas covering the cases when the frieze has four adjacent non-zero entries or two consecutive zero entries (we believe that the first lemma should also be true for an arbitrary width and arbitrary double zig-zag).
\begin{lem}\label{lemfour} Four non-zero entries in two consecutive columns, or in two consecutive diagonals, determine a unique tame symplectic $2$-frieze of width~$2$ over the complex numbers. This symplectic $2$-frieze is cluster-regular.
\end{lem}
\begin{proof}Since a tame symplectic 2-frieze is invariant under a glide symmetry it has the following general form:
\begin{gather}\label{tamew2}
\begin{array}{@{}cccccccccccccccccccc@{}}
\cdots & 1 & 1 & 1 & 1 & 1 & 1 & 1 & 1 & 1 & 1 & 1 & 1 &1 & 1 & 1 & \cdots & \\
\cdots & A & b &C&d&E&f&G&h&I&j&K&l&M&n&A&\cdots \\
 \cdots &h &I&j&K&l&M&n&A&b&C&d&E&f&G&h&\cdots \\
\cdots & 1 & 1 & 1 & 1 & 1 & 1 & 1 & 1 & 1 & 1 & 1 & 1 &1 & 1 & 1 & \cdots &
\end{array}
\end{gather}
where the upper case letters are black entries and lower case letterr white entries. Let us assume that the four entries $C$, $d$, $j$, $K$ are non zero complex numbers.
We know that there is at least one tame symplectic 2-frieze with these particular entries $C$, $d$, $j$, $K$ in two consecutuve columns. This is the frieze obtained by specializing the cluster $\chi=(x'_{1},x_{2},\mathbf{x_{3}}, \mathbf{x'_{4}})$ to $(d,j, C,K)$ in~\eqref{Fw} for $w=2$.

Let us show that there is at most one tame symplectic 2-frieze with these particular entries $C$, $d$, $j$, $K$ in two consecutuve columns. The frieze rule uniquely determines the entries $A$, $b$, $E$, $f$, $h$, $I$, $l$, $M$ of the frieze. The entries $G$ and $n$ are possibly undefined by the frieze rule (this happens in the case where $E=I=b=l=0$). However the tameness condition \eqref{C1} expressed in the $3\times 3$ minor centered at $E$ implies that $G$ is uniquely determined (one has $Gd-CM+1=E$ with $d\not=0$). Then $n=MA-G$ is also uniquely determined.
\end{proof}

\begin{lem}\label{lemtwo}Tame symplectic $2$-friezes of width $2$ over $\C$ containing two consecutive zero entries in at least one of its rows are exactly of the following form
\begin{gather}\label{tamefM}
\begin{array}{@{}cccccccccccccccccccc@{}}
\cdots & 1 & 1 & 1 & 1 & 1 & 1 & 1 & 1 & 1 & 1 & 1 & 1 &1 & \cdots & \\
\cdots & A & -1 &0&0&1&f&G&h&1&0&0&-1&M&\cdots \\
 \cdots &h &1&0&0&-1&M&n&A&-1&0&0&1&f&\cdots \\
\cdots & 1 & 1 & 1 & 1 & 1 & 1 & 1 & 1 & 1 & 1 & 1 & 1 &1 & \cdots &
\end{array}
\end{gather}
with $G=f+M$, $h=f+2M-1$, $n=-f-M^{2}$, $A=1-M$, for some free parameters $f, M\in \C$.
\end{lem}
\begin{proof}Let us consider the general tame symplectic 2-frieze \eqref{tamew2}. Assume $C=d=0$. The frieze rule immediately implies $j=K=0$ and then $E^{2}=-l=E$. So one immediately gets two possible values for $E$, namely $E=0$ or $E=1$. The tameness condition \eqref{C1} expressed in the $3\times 3$-minor centered at $C$ implies $IE-1=0$. Therefore $E\not=0$ so that $E=1$ and $l=-1$, $I=1$. Using the frieze rule one deduces $b=-1$, $G=f+M$, $n=-f-M^{2}$. Finally the tameness condition~\eqref{C1} expressed in the $3\times 3$ minor centered at $M$ gives $A=1-M$ from which one deduces $h=f+2M-1$ with the frieze rule. We have established that a tame symplectic $2$-friezes of width 2 containing two consecutive zero entries are necessarily of the form~\eqref{tamefM}. Let us show that the symplectic $2$-frieze~\eqref{tamefM} is tame for any values of the parameters~$f$ and~$M$. When $fM\not=0$ the frieze \eqref{tamefM} is exactly the one obtained by specializing the cluster $\chi=(x'_{1},x_{2},\mathbf{x_{3}}, \mathbf{x'_{4}})$ to $(f,-1,1,M)$ in~\eqref{Fw} for $w=2$. This means that the polynomials identities \eqref{C1}, \eqref{C2}, \eqref{C3}, required for the tameness of the frieze, hold for any non-zero values of $f$ and $M$, and therefore for any values of $f$ and $M$.
\end{proof}

We go back to the proof of Proposition~\ref{clustsing}(ii). We are looking for cluster-singular symplectic 2-friezes of width 2.
Let us consider a tame symplectic 2-frieze of width 2.

Case (1): there are two consecutive zeroes on the first (or last) row of the frieze. By Lemma~\ref{lemtwo} the array is of the form~\eqref{tamefM}. If $(f,M)\not=\big(0,\frac12\big)$ the frieze always contains four adjacent non-zero entries, so that it is cluster-regular according to Lemma~\ref{lemfour}.

Case (2): the first (or last) row has a sequence of the form $(x,0,y,0,z,0,t)$ where $xyzt\not=0$ are either black or white entries. The frieze rule will imply that the array is of the form
\[
\begin{array}{cccccccccccc}
 1 & 1 & 1 & 1 & 1 & 1 & 1 & 1 & & \\
x & 0 & y &0&z&0&t&r \\
 r &xy &s&yz&u&zt&v \\
 1 & 1 & 1 & 1 & 1 & 1 & 1 & 1 &
\end{array}
\]
with $rsuv\not=0$ so that the specified entries in the second row are non zero. By glide symmetry, one would have the entry $r$ next to $t$ in the first row so that $(t,zt,r,v)$ are non zero entries in two consecutive diagonals. So the frieze is cluster-regular according to Lemma \ref{lemfour}.

Case (3): the first row has two consecutive non zero entries $(x,y)$. Let us denote by $(s,t)$ the entries under $(x,y)$. If $st\not=0$ then one has two consecutive columns of non-zero entries and we are done by Lemma \ref{lemfour}. By the frieze rule one has $(s,t)\not=(0,0)$. We can now assume that
the array is locally of the form
\[
\begin{matrix}
 1 & 1 & 1 & 1 & \\
x & y &z& \\
0 &t &u&v& \\
 1 & 1 & 1 & 1
\end{matrix}
\]
with $t\not=0$. If $u\not=0$ then $(x,t,y,u)$ are non zero entries in two consecutive diagonals and we are done by Lemma \ref{lemfour}. If $u=0$ then $v\not=0$ (otherwise it goes back to Case (1)) and therefore $z=tv\not=0$. By induction we show that the next entries on the first row are non zero and the one of the second row alternate zero and non-zero, so that it goes back to Case (2).
\end{proof}

\begin{rem}To give a complete answer to the open problem \ref{open1} it remains to determine whether the following tame symplectic 2-frieze is cluster-regular or not:
\begin{gather*}
\begin{array}{@{}rrrrrrrrrrrrrrrrrrr@{}}
\cdots & 1 & 1 & 1 & 1 & 1 & 1 & 1 & 1 & 1 & 1 & 1 &1 & \cdots & \\
\cdots & -1 &0&0&1&0&\frac12&0&1&0&0&-1&\frac12&\cdots \\[1pt]
 \cdots &1&0&0&-1&\frac12&-\frac14&\frac12&-1&0&0&1&0&\cdots \\[1pt]
\cdots & 1 & 1 & 1 & 1 & 1 & 1 & 1 & 1 & 1 & 1 & 1 &1 & \cdots &
\end{array}
\end{gather*}
(this is the frieze \eqref{tamefM} for $(f,M)=\big(0,\frac12\big)$).
\end{rem}

\begin{rem}The frieze \eqref{sing1} is the point of coordinates $(a_{1}, \ldots, a_{6})=(0,\ldots, 0)$, $(b_{1}, \ldots, b_{6})=(-1, \ldots, -1)$ in the variety described in Example~\ref{exeqvar}. The tangent space at this point is of dimension~2 so that it is not a singular point of the variety (the variety is actually smooth).
\end{rem}

\begin{conj}The frieze of width $7$ given in \eqref{sing7} is cluster-singular.
\end{conj}

\section{Further investigations}\label{integ}

\subsection{Symplectic 2-friezes of positive integers}
A major problem in the theory of friezes is the classification of the friezes with positive integers. This problem usually leads to very nice combinatorial interpretations. The main result in the domain is the theorem of Conway--Coxeter \cite{CoCo} that classifies Coxeter friezes with triangulations of polygons. There are many results along the same line for the classification of other types of friezes, e.g., \cite{BaMa,BHJadv,FoPl,Tsc}.

There is also a nice extension of the Conway--Coxeter theorem in terms of 3d-dissections of polygons by allowing negative entries in the Coxeter friezes~\cite{Ovs3d}.

\begin{open}Find a combinatorial model for the classification of the symplectic $2$-friezes of positive integers.
\end{open}

Note that it is still an open problem in the case of classical $2$-friezes and $\SL$-friezes.

An ``easy'' way to produce symplectic $2$-friezes of positive integers is to use the maps $\phi_{\chi}$ given in Proposition~\ref{clustparam} and evaluate on the point $(z_{1}, \ldots, z_{2w})=(1, \ldots, 1)$. It was proved in~\cite{MGjaco} that when the cluster algebra $\A$ is of infinite type, i.e., has infinitely many clusters, this produces infinitely many friezes of positive integers. When the cluster algebra $\A$ is of finite type this procedure is in general not sufficient to produce all the friezes with positive integers (except in the case of Coxeter friezes).

 \begin{prop} There exist exactly $6$ symplectic $2$-friezes of width~$1$ with positive integers.
 \end{prop}

\begin{proof}The formal frieze $F_{1}$, cf.~\eqref{Fw} contains the 6 cluster variables of type $\mC_{2}$. It is known that in this case there is only 6 ways to get positive integers simultaneously for all the cluster variables~\cite{FoPl}.
\end{proof}

The 6 friezes are all obtained from \eqref{friw1} under the action of the diedral group~$D_{6}$.

\begin{conj}There exist exactly $112$ symplectic $2$-friezes of width~$2$ with positive integers.
\end{conj}

This conjecture is based on two evidences. First, Michael Cuntz kindly provide me with the~112 arrays of $\SL_{3}$-friezes which are symmetric with respect to the middle line. By Theorem~\ref{SLcar} they correspond to symplectic 2-friezes of width~2. The 112 $\SL_{3}$-friezes are grouped in 9 arrays modulo the action of the diedral group $D_{7}$. The second evidence is that the entries of the symplectic $2$-friezes of width 2 are cluster variables of type $\mF_{4}$ (see Example \ref{exF4}) and it has been already conjectured in~\cite{FoPl} that there are only 112 evaluations that make all the cluster variables positive integers. Note that in this case only 14 out of 28 cluster variables appear in the frieze, so that it was not necessary to have all the cluster variables simultaneously positive integers. Note also that there are 105 clusters of type $\mF_{4}$. Which would mean that~7 friezes are not obtained as $\phi_{\chi}(1,\ldots, 1)$.

Bernhard Keller mentioned to me that the applet~\cite{KellApp} allows to check that the cluster algebra of type $\mC_{2}\square\mA_{w}$ is cluster-infinite whenever $w\geq 3$. By consequence, there exist infinitely many symplectic $2$-friezes of positive integers of width~$w$ whenever $w\geq 3$.

\subsection[The 2-friezes of type $\mG_{2}$]{The 2-friezes of type $\boldsymbol{\mG_{2}}$}
Replacing the square by a cube in the local rule of the symplectic $2$-frieze, i.e., when the local rule in the array \eqref{symp2fri} reads
\begin{gather*}
\mathbf{AD-BC}=e, \qquad eh-fg=\mathbf{D}^{3},
\end{gather*}
leads to interesting arrays. In particular they are still periodic, of period $2(m+7)$ where $m$ is the width, and are invariant under a glide symmetry.
They are related to the combinatorics of the cluster algebras of type $\mG_{2}\times\mA_{m}$.
Such friezes could be called ``$2$-friezes of type $\mG_{2}$'' and it would be interesting to study them further from a combinatorial viewpoint and geometric viewpoint.
\begin{open}
What are the configurations spaces related to the $2$-friezes of type $\mG_{2}$?
\end{open}
We give below examples of $2$-friezes of type $\mG_{2}$ of width 1 and 2.
\begin{gather*}
\begin{array}{@{}ccccccccccccccccccc@{}}
\cdots&1&\mathbf{1}&1&\mathbf{1}&1&\mathbf{1}&1&\mathbf{1}&1&\mathbf{1}&1&\mathbf{1}&1&\mathbf{1}&1&\mathbf{1}&\cdots\\
\cdots&\mathbf{1}&2&\mathbf{3}&14&\mathbf{5}&9&\mathbf{2}&1&\mathbf{1}&2&\mathbf{3}&14 &\mathbf{5}&9&\mathbf{2}&1&\cdots\\
\cdots&1&\mathbf{1}&1&\mathbf{1}&1&\mathbf{1}&1&\mathbf{1}&1&\mathbf{1}&1&\mathbf{1}&1&\mathbf{1}&1&\mathbf{1}&\cdots
\end{array}\\
\begin{array}{@{}cccccccccccccccccccccc@{}}
\cdots&1&\mathbf{1}&1&\mathbf{1}&1&\mathbf{1}&1&\mathbf{1}&1&\mathbf{1}&1&\mathbf{1}&1&\mathbf{1}& 1 & \mathbf{1} & 1& \mathbf{1} &\cdots\\
\cdots&\mathbf{1}&2&\mathbf{3}&15&\mathbf{7}&28&\mathbf{6}&9&\mathbf{2}&1&\mathbf{1}&3&\mathbf{6}&77&\mathbf{14 } &36&\mathbf{3} &1&\cdots\\
\cdots&1&\mathbf{1}&3&\mathbf{6}&77&\mathbf{14}&36&\mathbf{3 } &1&\mathbf{1}&2&\mathbf{3}&15&\mathbf{7}&28&\mathbf{6}&9&\mathbf{2}&\cdots\\
\cdots&\mathbf{1}&1&\mathbf{1}&1&\mathbf{1}&1&\mathbf{1}&1&\mathbf{1}&1 & \mathbf{1} & 1&\mathbf{1}&1&\mathbf{1}&1&\mathbf{1}&1&\cdots
\end{array}
\end{gather*}
\subsection{Symplectic friezes in arbitrary dimensions}
In \cite{CoOv}, the spaces of Lagrangian configurations of $n$ lines in $\C^{2k}$ modulo $\Sp_{2k}$ are studied.
These configurations should give rise to more general symplectic friezes (recall that the case of symplectic 2-frieze is when $k=2$) and be connected to the cluster combinatorics of type $\mC_{k}\times \mA_{m}$, where $n=2k+m+1$. In particular the moduli space of Lagrangian configurations of $2k+2$ lines in $\C^{2k}$ should be a cluster variety of type $\mC_{k}$.
\begin{open}
What are the friezes related to Lagrangian configurations in arbitrary dimensions?
\end{open}

\appendix
\section{Desnanot--Jacobi identity}\label{DJid}
The Desnanot--Jacobi identity, or Dodgson formula, is a classical formula involving the determinant of an $(n\times n)$-matrix and its minors of order an $n-1$ and $n-2$ obtained by erasing the first and/or last row/column. The formula can be pictured as follows
\begin{gather*}
\begin{vmatrix}
*&*&*&*\\
*&*&*&*\\
*&*&*&*\\
*&*&*&*\\
\end{vmatrix}
\begin{vmatrix}
&&&\\
\;&*&*&\,\\
\;&*&*&\,\\
\;&&&\,\\
\end{vmatrix}
=
\begin{vmatrix}
\;*&*&*&\;\; \\
\;*&*&*&\;\;\\
\;*&*&*&\;\;\\
&&&\;\;\\
\end{vmatrix}
\begin{vmatrix}
\;&\;&&\\
\;&\;*&*&*\;\;\\
\;&\;*&*&*\;\;\\
\;&\;*&*&*\;\;\\
\end{vmatrix}
-
\begin{vmatrix}
&\;*&*&*\;\;\\
&\;*&*&*\;\;\\
&\;*&*&*\;\;\\
&\;&&\\
\end{vmatrix}
\begin{vmatrix}
&&&\;\\
\;*&*&*&\;\;\\
\;*&*&*&\;\;\\
\;*&*&*&\;\;\\
\end{vmatrix},
\end{gather*}
where the deleted columns/rows are left blank.

It is a key identity to establish Proposition~\ref{3344}.

\section[Properties of $\SL_{k+1}$-friezes]{Properties of $\boldsymbol{\SL_{k+1}}$-friezes}\label{apSL}

\subsection[Definition of tame $\SL_{k+1}$-friezes]{Definition of tame $\boldsymbol{\SL_{k+1}}$-friezes}\label{tameSL}
An $\SL_{k+1}$-{\it frieze} is an array of numbers consisting in a finite number of infinite rows:
\begin{gather}\label{FREq}
\begin{array}{ccccccccccccc}
&&&& \vdots&&&& \vdots&&&\\
&0&&0&&0&&0&&0&&\ldots\\[2pt]
\ldots&&1&&1&&1&&1&&1&\\[2pt]
&\ldots&&\;d_{0,w-1}&&\;d_{1,w}&&\;d_{2,w+1}&&\ldots&&\ldots\\
&&&\! \iddots&& \iddots&& \iddots&&&&\\
\ldots&& d_{0,1}&&d_{1,2}&&d_{2,3}&&d_{3,4}&&d_{4,5}&\\[2pt]
& d_{0,0}&&d_{1,1}&&d_{2,2}&&d_{3,3}&&d_{4,4}&&\ldots\\[2pt]
\ldots&&1&&1&&1&&1&&1&\\[2pt]
&0&&0&&0&&0&&0&&\ldots\\
&&&& \vdots&&&& \vdots&&&
\end{array}
\end{gather}
where the strip is bounded by $k$ rows of 0's at the top, and at the bottom, and where every ``diamond'' $(k+1)\times (k+1)$-subarray forms an element of $\SL_{k+1}$. The number of rows between the bounding rows of~$1$'s is called the {width} and is denoted by~$w$.

More precisely, the entries in the array are denoted by $(d_{i,j})$, with $i,j\in\Z$ such that
\begin{gather*}
i-k-1\leq{}j\leq{}i+w+k.
\end{gather*}

The array is a tame $\SL_{k+1}$-frieze of width $w$ when it satisfies:
\begin{itemize}\itemsep=0pt
\item ``boundary conditions''
\begin{gather*}
\begin{cases}
d_{i,i-1} = d_{i,i+w} = 1 & \text{for all}\ i,\\
d_{i,j} = 0& {\rm for} \ i-k-1\leq j<i-1\ {\rm or}\ i+w<j\leq i+w+k.
\end{cases}
\end{gather*}
\item ``$\SL_{k+1}$-conditions'' on the $(k+1)\times (k+1)$-adjacent minors
\begin{gather*}
\left\vert
\begin{matrix}
d_{i,j}&d_{i,j+1}&\ldots&d_{i,j+k}\\
d_{i+1,j}&d_{i+1,j+1}&\ldots&d_{i+1,j+k}\\
\ldots& \ldots&& \ldots\\
d_{i+k,j}&d_{i+k,j+1}&\ldots&d_{i+k,j+k}
\end{matrix}
\right\vert=1,
\end{gather*}
for all $(i,j)$ in the index set.
\item ``tameness conditions'' on the $(k+2)\times (k+2)$-adjacent minors
\begin{gather*}
\left\vert
\begin{matrix}
d_{i,j}&d_{i,j+1}&\ldots&d_{i,j+k+1}\\
d_{i+1,j}&d_{i+1,j+1}&\ldots&d_{i+1,j+k+1}\\
\ldots& \ldots&& \ldots\\
d_{i+k+1,j}&d_{i+k+1,j+1}&\ldots&d_{i+k+1,j+k+1}
\end{matrix}
\right\vert=0,
\end{gather*}
for all $(i,j)$ in the index set.
\end{itemize}

\subsection[Tame $\SL_{k+1}$-friezes and difference equations]{Tame $\boldsymbol{\SL_{k+1}}$-friezes and difference equations}
An $n$-superperiodic difference equation of order $k+1$ is a system
\begin{gather}\label{REq}
V_{i}=a_{i}^1V_{i-1}-a_{i}^{2}V_{i-2}+ \cdots+(-1)^{k-1}a_{i}^{k}V_{i-k}+(-1)^{k}V_{i-k-1},
\end{gather}
where $a_{i}^{j}\in \C$, with $i\in\Z$ and $1\leq j\leq k$, are given coefficients (note that the superscript $j$ is an index, not a power)
and $V_i$ are unknowns, such that
\begin{itemize}\itemsep=0pt
\item the coefficients are $n$-periodic, i.e., $a_{i+n}^{j}=a_{i}^{j}$, for all $i$, $j$.
\item every solution $(V_{i})_{i\in \Z}$ is $n$-(anti)periodic, i.e., $V_{i+n}=(-1)^{k} V_{i}$, for all $i\in \Z$.
\end{itemize}

\begin{prop}[\cite{MGOST}]\label{frieq}
An array $(d_{i,j})$ as \eqref{FREq} forms a tame $\SL_{k+1}$ friezes if and only if every diagonal
\begin{gather*}
d_{i,\bullet}=(0, \ldots, 0,1, d_{i,i}, d_{i,i+1}, \ldots, d_{i,i+w-1},1)
\end{gather*}
gives a solution of a same $(w+k+2)$-superperiodic equation of the form \eqref{REq} with initial conditions
\begin{gather*}
(V_{i-k-1},V_{i-k},\ldots,V_{i-1})=(0,0,\ldots,0,1).
\end{gather*}
\end{prop}

\subsection[Coefficients of the difference equations from the entries of the frieze and \textit{vice versa}]{Coefficients of the difference equations from the entries of the frieze \\and \textit{vice versa}}

We collect some formulas given in \cite[Sections~5.2 and 5.3]{MGOST}.

Let \eqref{REq} be the associated equation associated to a tame $\SL_{k+1}$-frieze \eqref{FREq}. The coeffi\-cient~$a_{i-1}^{k-j}$ of the equation can be expressed as a $(j+1)\times (j+1)$-minor of adjacent entries in the frieze:
\begin{gather}\label{DuDeT}
a_{i-1}^{k-j}=
\left|
\begin{matrix}
d_{i+1,i+w}&1&\\
\vdots&\ddots& 1\\
d_{i+j+1,i+w}&\cdots&d_{i+j+1,i+j+w}
\end{matrix}
\right|.
\end{gather}

Conversely, the entries of the frieze can be computed as determinant of matrices involving the coefficients of the equation. The entry $d_{i, i+j}$ can be computed using
$(j+1)\times (j+1)$-determinants
\begin{gather}
\label{DetEq1}
d_{i,i+j} = \left|
\begin{matrix}
a_{i}^1&1&\\
\vdots&a_{i+1}^1&1&\\
a_{i+k-1}^{k}&&\ddots&\ddots\\
1&&&\ddots&\ddots\\
&\ddots&&&a_{i+j-1}^1&1\\
&&1&a_{i+j}^{k}&\ldots&a_{i+j}^1
\end{matrix}
\right|,
\end{gather}
or alternatively, using $(w-j)\times (w-j)$-determinants
\begin{gather}\label{DetEq2}
d_{i,i+j}=
\left|
\begin{matrix}
a_{i-w+j-1}^k&\ldots&a_{i-w+j-1}^1& 1&&\\
1&a_{i-w+j}^k&\ldots&a_{i-w+j}^1& 1&\\
&1&\ldots&&\ddots& 1\\
&&\ddots&&\ddots&\vdots\\
&&& 1&a_{i-3}^k&a_{i-3}^{k-1}\\
&&&& 1&a_{i-2}^k
\end{matrix}
\right|.
\end{gather}

\subsection{Projective duality}
Given a tame $\SL_{k+1}$-frieze $(d_{i,j})$, its projective dual is defined as the array consisting of $k\times k$ adjacent minors:
\begin{gather*}
d_{i,j}^{*}=
\left\vert
\begin{matrix}
d_{i,j}&d_{i,j+1}&\ldots&d_{i,j+k-1}\\
d_{i+1,j}&d_{i+1,j+1}&\ldots&d_{i+1,j+k-1}\\
\vdots& && \vdots\\
d_{i+k-1,j}&d_{i+k-1,j+1}&\ldots&d_{i+k-1,j+k-1}
\end{matrix}
\right\vert.
\end{gather*}

\begin{prop}[\cite{MGOST}]\quad
\begin{enumerate}\itemsep=0pt
\item[$(i)$] The projective dual frieze to a tame $\SL_{k+1}$-frieze is a tame $\SL_{k+1}$-frieze.
\item[$(ii)$] If $a_{i}^{j}$ are the coefficients of the equation \eqref{REq} associated to a $\SL_{k+1}$-frieze, then
\begin{gather}\label{DualREq}
V^*_{i}=a_{i+k-1}^kV^*_{i-1}-a_{i+k-2}^{k-1}V^*_{i-2}+ \cdots+(-1)^{k-1}a_{i}^1V^*_{i-k}+(-1)^{k}V^*_{i-k-1}
\end{gather}
is the superperiodic equation associated to the projective dual.
\item[$(iii)$] The projective dual to a tame $\SL_{k+1}$-frieze is just the symmetric array with respect to the median horizontal axis.
\end{enumerate}
\end{prop}

Let us make the above statement number (iii) more precise.
\begin{prop}
A tame $\SL_{k+1}$-frieze of width $w$ satisfies
\begin{gather}\label{dijminor}
d_{i,j}=\left\vert
\begin{matrix}
d_{j-w-k,i-k-1}&\ldots&d_{j-w-k,i-2}\\
\vdots& & \vdots\\
d_{j-w-1,i-k-1}&\ldots&d_{j-w-1,i-2}
\end{matrix}
\right\vert
=d^{*}_{j-w-k,i-k-1}
\end{gather}
for all $i$, $j$.
 \end{prop}

 \begin{proof}First, note the diagonals
\begin{gather*}
d_{\bullet,i}=(0, \ldots, 0,1, d_{i,i}, d_{i-1,i}, \ldots, d_{i-w+1,i},1)
\end{gather*}
satisfy the dual equation \eqref{DualREq}. Let us consider two sequences $V=(V_{\ell})_{1\leq \ell \leq j-i+k+2}$ and $W=(W_{\ell})_{1\leq \ell \leq j-i+k+2}$ forming blocks of size $(k+1)\times (j-i+k+2)$ in the frieze that differ by a shift along the north-west diagonal:
\begin{gather*}
\begin{blockarray}{@{\,}c@{\,}c@{\,}c@{\,}c@{\,}c@{\,}c@{\,}c@{\,}c@{\,}c@{\,}c@{\,}c@{\,}c@{\,}c@{\,}}
V_{1} & V_{2} & \cdots &V_{k+1} & V_{k+2} & &\cdots\qquad &V_{j-i+1}&&&V_{j-i+k+2} \\[2pt]
\begin{block}{(@{}c@{\,}c@{\,}c@{\,}c@{\,}c@{\,}c@{\,}c@{\,}c@{\,}c@{\,}c@{\,}c@{\,})@{\,}c@{\,}c@{\,}}
d_{j-w-k,i-k-1} &d_{j-w-k,i-k} &\cdots&\cdots&d_{j-w-k,i}&&\cdots&1&0&&0\\[2pt]
\vdots&\vdots&&&\vdots&&\cdots&&1&&& \ \ =&V,\\[2pt]
 \vdots &\vdots& &&\vdots&&\cdots&&&\ddots&0\\[2pt]
d_{j-w,i-k-1}&d_{j-w,i-k}&\cdots&\cdots&d_{j-w,i}&&\cdots&&\cdots&& 1\\[2pt]
\end{block}
\vdots &\vdots& & & \vdots & & &&&&\vdots\\[2pt]
\begin{block}{(@{\,}c@{\,}c@{\,}c@{\,}c@{\,}c@{\,}c@{\,}c@{\,}c@{\,}c@{\,}c@{\,}c@{\,})@{\,}c@{\,}c@{\,}}
1&d_{i-k,i-k}&\cdots&\cdots&d_{i-k,i}&&\cdots&&\cdots&&d_{i-k,j}\\[2pt]
0&1&\ddots&&\vdots&&&&&&\vdots & \ \ =& W.\\[2pt]
 \vdots &\vdots&\ddots &\ddots&\vdots&&&&&&\vdots\\[2pt]
0&0&\cdots&1&d_{i,i}&&\cdots&&\cdots&& d_{i,j}\\[2pt]
\end{block}
W_{1} & W_{2} & \cdots &W_{k+1} & W_{k+2} & &\cdots\qquad &&\cdots&&W_{j-i+k+2} \\[2pt]
\end{blockarray}
\end{gather*}
To simplify the picture we have rotated the frieze by 45 degree so that the shift between the blocks is now seen vertically.

Due to the recurrence relation in the frieze, the matrix $V$ can be obtained from $W$ by multiplying on the left by a product of matrices of the form
\[
(-1)^{k}
\left(
\begin{matrix}
(-1)^{k}a^{k}_{\ell+k-1} & \ldots&-a_{\ell}^{1}&1 \\
 1 & & \\
 & \ddots & \\
 &&1&0\\
\end{matrix}
\right).
\]
So that we have $V=TW$ with $V\in \SL_{k+1}$. This implies that the minors of order $(k+1)$ in $V$ and $W$ involving the same columns are equal. In particular, we have
\begin{gather}\label{formul}
d_{i,j}=\det(W_{1}, \ldots, W_{k}, W_{j-i+k+2} )=\det(V_{1}, \ldots, V_{k}, V_{j-i+k+2} )\\
\hphantom{d_{i,j}}{} = \left\vert
\begin{matrix}
d_{j-w-k,i-k-1}&\ldots&d_{j-w-k,i-2}\\
\vdots& & \vdots\\
d_{j-w-1,i-k-1}&\ldots&d_{j-w-1,i-2}
\end{matrix}
\right\vert.\tag*{\qed}
\end{gather}\renewcommand{\qed}{}
\end{proof}

\subsection{Gale duality}\label{friezeG}
The Gale duality is a one-to-one correspondence between $\SL_{k+1}$-frieze of width $w$ and $\SL_{w+1}$-frieze of width~$k$. The correspondence is given by the following result.
\begin{prop}[{\cite[Proposition~4.1.3]{MGOST}}]\label{GDFprop} Consider a $\SL_{k+1}$-frieze of width $w$, and its associated equation~\eqref{REq}. The following array formed by the coefficients of the equation
\begin{gather*}
\begin{array}{@{}ccccccccccccccccccccccccccc@{}}
&\ldots&1&&1&&1&&1&&1&&1\\
&&&\ldots&& a^1_n&&a^1_1&&a^1_2&& \ldots&& a^1_n&\\
&&&& a^2_n&&a^2_1&& a^2_2&&&&a^2_n&\\
&\ldots&& \iddots && \iddots&& \iddots&&&& \iddots&&\ldots\\
&&a^k_n&& a^k_1&&a^k_2&& \ldots&& a^k_n&&\ldots\\
&1&&1&&1&&1&&1&&1&&\ldots
\end{array}
\end{gather*}
is an $\SL_{w+1}$-frieze of width $k$.
\end{prop}

\subsection*{Acknowledgements}
I am deeply grateful to Valentin Ovsienko for sharing with me ideas and results of the preliminary version of~\cite{CoOv}. I am also grateful to Michael Cuntz for computer calculations and Bernhard Keller for help with the applet~\cite{KellApp}. I also want to thank Luc Pirio for stimulating discussions on the subject. This work is supported by the ANR project $SC^{3}\!A$, ANR-15-CE40-0004-01.

\pdfbookmark[1]{References}{ref}
\LastPageEnding

\end{document}